\documentclass[11pt,reqno]{amsart}
\usepackage{amssymb}
\usepackage{mathptmx}
\usepackage{mathtools}
\usepackage{mathrsfs}
\usepackage[export]{adjustbox}
\usepackage[all]{xy}
\usepackage{stmaryrd}
\usepackage{fancyhdr}
\usepackage{hyperref}
\usepackage{tikz-cd}
\usepackage{cite}
\usepackage{amsmath,amsfonts}
\usepackage{graphicx}
\usepackage{wrapfig}
\usepackage{array}
\usepackage{amsthm}
\usepackage[left=1.2in, right=1.2in, top=1in, bottom=1in]{geometry}
\usepackage{etoolbox}
\patchcmd{\section}{\scshape}{\bfseries}{}{}
\makeatletter
\renewcommand{\@secnumfont}{\bfseries}
\makeatother
\xyoption{all}
\usepackage{secdot}

\theoremstyle{plain}

\newcommand{\FF}{\mathbb{F}}    
\newcommand{\NN}{\mathbb{N}} 

\newcommand{\wild}{\mathbb{L}}
\newcommand{\NI}{\operatorname{NI}} 
 
\newcommand{\Rep}{\operatorname{Rep}}  

\newcommand{\nil}{\operatorname{nil}}
\newcommand{\Hom}{\operatorname{Hom}}

\input xy
\xyoption{all}

\theoremstyle{definition}
\newtheorem{mydef}{\textbf{Definition}}[section]
\newtheorem{myeg}[mydef]{\textbf{Example}}

\newtheorem{question}[mydef]{\textbf{Question}}

\newtheorem{rmk}[mydef]{\textbf{Remark}}
\newtheorem{construction}[mydef]{\textbf{Construction}}

\theoremstyle{plain}

\newtheorem*{nothma}{\textbf{Theorem A}}
\newtheorem*{nothmb}{\textbf{Theorem B}}
\newtheorem*{nothmc}{\textbf{Theorem C}}

\newtheorem{mythm}[mydef]{\textbf{Theorem}}
\newtheorem{lem}[mydef]{\textbf{Lemma}}
\newtheorem{pro}[mydef]{\textbf{Proposition}}

\newtheorem{cor}[mydef]{\textbf{Corollary}}

\tikzset{main node/.style={circle,fill=black,draw,minimum size=0.3cm,inner sep=0pt},
}

\begin{document}

	\title{On the combinatorics of $\FF_1$-representations of pseudotree quivers}

	\author{Jaiung Jun}
	\address{State University of New York at New Paltz, NY, USA}
	\curraddr{}
	\email{junj@newpaltz.edu}
	
	\author{Jaehoon Kim}
	\address{Department of Mathematical Sciences, KAIST, South Korea.}
	\curraddr{}
	\email{jaehoon.kim@kaist.ac.kr}
	
	\author{Alex Sistko}
	\address{Manhattan College, NY, USA}
	\curraddr{}
	\email{asistko01@manhattan.edu}

	\makeatletter
	\@namedef{subjclassname@2020}{%
		\textup{2020} Mathematics Subject Classification}
	\makeatother
	

	\maketitle

	
\begin{abstract}
We investigate quiver representations over $\FF_1$. Coefficient quivers are combinatorial gadgets equivalent to $\FF_1$-representations of quivers. We focus on the case when the quiver $Q$ is a pseudotree. For such quivers, we first use the notion of coefficient quivers to provide a complete classification of asymptotic behaviors of indecomposable representations over $\FF_1$. Then, we prove some fundamental structural results about the Lie algebras associated to pseudotrees. Finally, we construct examples of $\FF_1$-representations $M$ of a quiver $Q$ by using coverings, under which the Euler characteristics of the quiver Grassmannians $\textrm{Gr}^Q_{\underline{d}}(M)$ can be computed in a purely combinatorial way. 
\end{abstract}

	
\section{Introduction}

Quivers are finite directed graphs, and a representation of a quiver $Q$ over a field $k$ is a collection $\mathbb{V}=\{V_i,f_\alpha\}_{i \in Q_0,\alpha \in Q_1}$ of $k$-vector spaces $V_i$ and $k$-linear maps $f_\alpha$, where $Q_0$ (resp.~$Q_1$) is the set of vertices (resp.~arrows). Quivers and their representations are natural objects which arise in various areas of mathematics. For instance, a quiver Grassmannian is a certain moduli space of quiver representations generalizing a Grassmannian (which correspond to the case when the quiver has a single vertex and no arrows). Quiver Grassmannians are projective varieties, and in \cite{reineke2013every}, Reineke proved that any projective variety is indeed a quiver Grassmannian.

The notation $\FF_1$ (``the field with one element'') denotes an idea than a well-defined mathematical object. Depending on literature, algebraic geometry over $\FF_1$ could mean several things. Sometimes, it refers to algebraic geometry over monoids (hence the theory of monoid schemes, for instance, by  Deitmar \cite{deitmar2008f1,Deitmar}, Soul\'e \cite{soule2004varietes}, Connes-Consani-Marcolli \cite{connes2009fun,con1,con2}, and Corti\~{n}as-Haesemayer-Walker-Weibel \cite{cortinas2015toric}). Sometimes it means algebraic geometry over semirings; in the case of idempotent semirings, it makes its connection to tropical geometry \cite{giansiracusa2016equations}. Over the natural numbers $\mathbb{N}$, algebraic geometry over $\mathbb{N}$ brings positivity at a very elementary level \cite{borger2016witt}.

Szczesny \cite{szczesny2011representations} first introduced a notion of quiver representations over $\FF_1$ by considering it as a degenerated combinatorial model of the category $\textrm{Rep}(Q,\FF_q)$ of representations of $Q$ over $\FF_q$; $\FF_q$-vector spaces are replaced by $\FF_1$-vector spaces (finite pointed sets), and $\FF_q$-linear maps are replaced by $\mathbb{F}_1$-linear maps (pointed set maps $f:(V,0_V) \to (W,0_W)$ such that $f|_{V\backslash f^{-1}(\{0_W\})}$ is injective). Szczesny proved that the category $\textrm{Rep}(Q,\FF_1)$ satisfies several nice properties. He further studied the Hall algebra associated to $\textrm{Rep}(Q,\FF_1)$ and its relation to Kac-Moody algebras by seeing it as a degenerate version (at $q=1$) of the theorems by Ringel \cite{ringel1990hall} and Green \cite{green1995hall}.

Despite its finite and combinatorial flavor, one can always obtain functorially a representation of $Q$ over a field $k$ from a representation of $Q$ over $\FF_1$: a finite pointed set $(V,0_V)$ defines a vector space $V_k$ whose basis is $V-\{0_V\}$. This induces a functor which is faithful (but not full in general):
\begin{equation}\label{eq: base change}
	k\otimes_{\FF_1}-:\textrm{Rep}(Q,\FF_1) \to \textrm{Rep}(Q,k).
\end{equation}

In \cite{jun2020quiver, jun2021coefficient}, the first and third authors introduced a notion of coefficient quivers for representations of a quiver over $\FF_1$, and interpreted representations $\mathbb{V} \in \textrm{Rep}(Q,\FF_1)$ as certain quiver maps $\Gamma_{\mathbb{V}} \to Q$, where $\Gamma_{\mathbb{V}}$ is a quiver constructed from $\mathbb{V}$. The quiver $\Gamma_{\mathbb{V}}$ (along with its structure map to $Q$) is said to be a \emph{coefficient quiver} of $Q$ due to its resemblance to the coefficient quivers defined by Ringel \cite{ringel1998exceptional}. The notion of coefficient quivers is our key ingredient to translate various questions on quiver representations of $Q$ over $\FF_1$ (and its extension to $\FF_q$ or $\mathbb{C}$) into purely combinatorial questions. 

\begin{myeg}\label{example: colored quiver}
	Consider the following quiver:
	\[
	Q=	\left( \begin{tikzcd}
		\bullet \arrow[loop left,looseness=15,"\alpha_1"]
		\arrow[loop right, looseness=15,"\alpha_2"]
	\end{tikzcd} \right)
	\]
	Let $\mathbb{V}=\{V_0,f_{\alpha_1},f_{\alpha_2}\}$ be an $\FF_1$-representation of $Q$, where $V_0=\{0_{V_0},1,2,3\}$, $f_{\alpha_1}(n)=\max\{0,n-1\}$, and $f_{\alpha_2}(n)=\max\{0,n-2\}$. Then, one can encode $\mathbb{V}$ as a ``colored quiver'' $\Gamma_\mathbb{V}$. The vertices are nonzero elements of $V_0$ and the arrows tell us where each element is mapped to by $f_{\alpha_1}$ (in dashed blue) and $f_{\alpha_2}$ (in red):
	\[
	\Gamma_\mathbb{V}=\left(\begin{tikzcd}[arrow style=tikz,>=stealth,row sep=1.5em, column sep=1em]
		1 
		&& 2 \arrow[ll,swap, dashed, blue]\\
		& 3 
		\arrow[ul, red]
		\arrow[ur,dashed, blue]
	\end{tikzcd} \right)
	\]
	The coloring of $\Gamma_\mathbb{V}$ can be encoded as the structure map to $Q$ as follows:
	\begin{equation}\label{eq: structure map}
		F: \quad \left(\begin{tikzcd}[arrow style=tikz,>=stealth,row sep=1.5em,column sep=1em]
			1 
			&& 2 \arrow[ll,swap,"B_1"]\\
			& 3 
			\arrow[ul,"R"]
			\arrow[ur,swap,"B_2"]
		\end{tikzcd} \right)   \longrightarrow  \left( \begin{tikzcd}
			\bullet \arrow[loop left,looseness=15,"\alpha_1"]
			\arrow[loop right, looseness=15,"\alpha_2"]
		\end{tikzcd} \right)
	\end{equation}
	where all vertices of $\Gamma_\mathbb{V}$ map to the vertex of $Q$, $R$ maps to $\alpha_2$, and $B_1,B_2$ map to $\alpha_1$. 
\end{myeg}

In fact, our motivation for studying quiver representations over $\FF_1$ is in line of Tits' initial motivation in search for $\FF_1$  \cite{tits1956analogues}; finding some combinatorial cores (if exist) of geometric structures. To be precise, we aim to develop theory of quiver representations over $\FF_1$ to approach problems for quiver representations over $\mathbb{C}$ (or $\FF_q$) via the base change functor \eqref{eq: base change}. For instance, in \cite{jun2021coefficient}, we employ the techniques in \cite{irelli2011quiver} and \cite{haupt2012euler} to compute Euler characteristics of quiver Grassmannians by simply counting (successor-closed) subquivers of a coefficient quiver. In \cite{jun2021coefficient}, we prove that for several classes of quiver Grassmannians, one can compute Euler characteristic in a purely combinatorial way by ``counting $\FF_1$-rational points''. Note that our results strictly include tree and band modules. See Example \ref{example: not band}.
\medskip

In \cite{jun2020quiver}, inspired by tame-wild dichotomy in representation theory of finite dimensional algebras, the first and third authors introduced a growth function counting isomorphism classes of indecomposable, nilpotent representations (Definition \ref{definition: indecomposables growth}). This growth function defines an order $\leq_{nil}$ on quivers, and hence define an equivalence relation $\approx_{\nil}$ on quivers. It was proved in \cite{jun2020quiver} that trees, the one loop quiver $\wild_1$, and two loop quiver $\wild_2$ belong to different equivalence classes. More precisely, the trees form the equivalence class of connected quivers with finite representation type over $\mathbb{F}_1$, and $\wild_1$ along with the type $\tilde{\mathbb{A}}$ affine Dynkin quivers form the equivalence class of all connected quivers with bounded representation type over $\mathbb{F}_1$. Furthermore, for any quiver $Q$ one has $Q \leq_{nil} \wild_2$ and $Q\approx_{\nil} \wild_2$ if $\operatorname{rank}H_1(Q,\mathbb{Z}) \geq 2$. What remained was the unknown territory of proper pseudotrees, i.e. the connected $Q$ that are not of type $\tilde{\mathbb{A}}$ but which still satisfy $H_1(Q,\mathbb{Z})\cong \mathbb{Z}$. It was known that such quivers were all strictly greater than $\wild_1$ and bounded above by $\wild_2$: it was not clear whether any were equivalent to one another, or to $\wild_2$.

In Section \ref{section: growth}, we completely classify quivers upto the equivalence relation $\approx_{\nil}$. We do this by recursively counting the coefficient quivers $\Gamma$ (of a fixed number of vertices) of a proper pseudotree $Q$. As a result, we prove the following.

\begin{nothma}[Corollary \ref{corollary: main cor}]
Let $Q$ be a proper pseudotree. Then 
\[
Q\approx_{\nil} Q', \textrm{ where } 	Q'=\begin{tikzcd}
	\bullet \arrow[r] &	\bullet 
	\arrow[out=30,in=-30,loop]
\end{tikzcd}
\]
\end{nothma}

Next, we study the Lie algebras associated to $\mathbb{F}_1$-representations of $Q$. Classically, one attaches a Hall algebra to a finitary abelian category. Dyckerhoff and Kapranov \cite{dyckerhoff2012higher} generalized this construction to proto-exact categories. In particular, one can obtain the Hall algebra of the category $\Rep(Q,\FF_1)$ of $\FF_1$-representations of $Q$. See Section \ref{subsection: Hall algebras} for details. In this case, the Hall algebra of $\Rep(Q,\FF_1)$ is the enveloping algebra of the Lie algebra defined by indecomposable objects. 

We say that a connected quiver $Q$ is a pseudotree (not necessarily proper) if the rank of $H_1(Q,\mathbb{Z})$ is at most one. In other words, a pseudotree is either a tree, a type $\tilde{\mathbb{A}}$ quiver, or a proper pseudotree as defined above. As a direct consequence of the results in \cite{jun2021coefficient}, the coefficient quiver of any nilpotent, indecomposable representation of a pseudotree $Q$ is again a pseudotree. Let $\mathcal{T}_Q(n)$ (resp.~$\mathcal{P}_Q(n)$) be the set of isomorphism classes of $n$-dimensional indecomposable objects in $\Rep(Q,\FF_1)_{\nil}$ (nilpotent representations) whose coefficient quiver is a tree (resp.~a pseudotree, but not a tree). Let $\mathcal{T}_Q := \bigcup_{n \in \mathbb{N}}{\mathcal{T}_Q(n)}$ and  $\mathcal{P}_Q := \bigcup_{n \in \mathbb{N}}{\mathcal{P}_Q(n)}$. Note that in \cite{jun2021coefficient}, the first and third authors studied the Lie algebras associated to trees and affine Dynkin quivers of type $\tilde{\mathbb{A}}$. The following theorem extends these results to proper pseudotrees.

\begin{nothmb}[Theorem \ref{theorem: Lie main thm}]
	Let $Q$ be a pseudotree. Let $\mathfrak{n}$ be the Lie algebra of nilpotent indecomposable representations of $Q$, and $\mathfrak{p}=\textrm{Span}(\mathcal{P}_Q)$. Then the following hold: 
\begin{enumerate} 
	\item As a Lie algebra, $\mathfrak{n}$ is generated by $\mathcal{T}_Q$. 
	\item
$\mathfrak{p}$ is a Lie ideal of $\mathfrak{n}$. As a Lie algebra, the isomorphism class of $\mathfrak{n}/\mathfrak{p}$ only depends on the underlying undirected graph of $Q$. 
\end{enumerate}
\end{nothmb}

Finally, we turn our attention to coverings of quivers. In \cite{jun2021coefficient}, the first and third authors proved that when the coefficient quiver $\Gamma_M$ of an $\FF_1$-representation $M$ is equipped with certain integer-valued functions on vertices satisfying some nice conditions (Definitions \ref{definition: winding},\ref{definition: disting}), one can compute the Euler characteristic of a quiver Grassmannian in a purely combinatorial way by counting the number of subquivers of the coefficient quiver of $M$ (Theorem \ref{theorem: quiver grass thm}). In Section \ref{section: coverings and contractions}, we explore this idea further by considering coverings and contractions of coefficient quivers. In particular, we prove the following among others.

\begin{nothmc}[Proposition \ref{proposition: main proposition}]
Let $Q$ be a quiver. For each $e \in Q_1$, one can construct a covering $c_e:\Gamma_e \to Q$ which has a nice grading distinguishing vertices. 
\end{nothmc}

This paper is organized as follows. In Section \ref{section: preliminaries}, we recall basic definitions and facts for $\FF_1$-representations of quivers. In Section \ref{section: growth}, we prove our Theorem A on the asymptotic behavior of a growth function counting isomorphism classes of indecomposable, nilpotent representations of pseudotrees. In Section \ref{section: Lie algebra}, we prove Theorem B on structural results on Lie algebras of $\FF_1$-representations of pseudotrees. Finally, in Section \ref{section: coverings and contractions}, we explore a notion of coverings, and prove Theorem C. 

\bigskip

\textbf{Acknowledgment}\hspace{0.1cm} Part of this research was done while the first author was visiting the second author at KAIST. The first author thanks KAIST for the hospitality.

\section{Preliminaries} \label{section: preliminaries}

\subsection{$\FF_1$-representations of quivers}

Representations of a quiver over $\FF_1$ is obtained by replacing the category of vector spaces with the category of ``$\FF_1$-vector spaces.'' 

\begin{mydef}
An $\FF_1$-vector space is a finite pointed set $(V,0_V)$. The dimension of $V$ is $\dim_{\FF_1}V=|V|-1$. An $\FF_1$-linear map $$f:(V,0_V)\to (W,0_W)$$ between $\FF_1$-vector spaces is a pointed function $f$ whose restriction to $(V-f^{-1}(0_W))$ is injective. We call $\textrm{Vect}(\mathbb{F}_1)$, the category of \emph{vector spaces} over $\mathbb{F}_1$ whose objects are $\FF_1$ vector spaces and morphism are $\FF_1$-linear maps. 

\end{mydef}

\begin{mydef}
	A \emph{quiver} $Q=(Q_0,Q_1,s,t)$ is a finite directed graph\footnote{We allow multiple arrows and loops}, where
	\begin{enumerate}
		\item 
		$Q_0$ (resp.~$Q_1$) is the finite set of vertices (resp.~arrows),
		\item
		$s$ and $t$ are functions
		\[
		s,t:Q_1 \to Q_0
		\] 
		assigning to each arrow in $Q_1$ its \emph{source} and \emph{target} in $Q_0$. 
	\end{enumerate}
\end{mydef}

Let $Q$ and $Q'$ be quivers. A \emph{quiver map} $c : Q \rightarrow Q'$ is a pair of functions  
\[ 
c_i : Q_i \rightarrow Q_i', \quad \textrm{for }i = 0,1,
\] 
satisfying the following condition:
\[
s(c_1(\alpha)) = c_0(s(\alpha)) \textrm{ and } t(c_1(\alpha)) = c_0(t(\alpha)), \quad \forall \alpha \in Q_1.
\] 
A quiver map $c$ is injective (resp.~surjective) if and only if $c_0$ and $c_1$ are injective (resp.~surjective).   

If $Q$ is connected, its underlying graph $\overline{Q}$ can be considered a connected $1$-simplex. By a slight abuse of notation, we let $H_1(Q,\mathbb{Z})$ stand for the first homology group (with integer coefficients) of $\overline{Q}$. Note that $H_1(Q,\mathbb{Z})$ is a finitely-generated free abelian group. 

\begin{mydef} 
Let $Q$ be a connected quiver. Then $Q$ is a \emph{pseudotree} if $\operatorname{rank}H_1(Q,\mathbb{Z}) \le 1$. We say that $Q$ is a \emph{proper} pseudotree if $Q$ is a pseudotree which is neither a tree nor a type $\tilde{\mathbb{A}}$ quiver (i.e. an arbitrary orientation of a simple cycle).
\end{mydef} 

The class of proper pseudotrees will be the main focus of this article. In particular, we will be interested in their $\FF_1$-representations, as defined below.

\begin{mydef}\cite[Definition 4.1]{szczesny2011representations}\label{definition: representation of a quiver over $F_1$}
Let $Q$ be a quiver. An $\FF_1$-representation of $Q$ is a collection $\mathbb{V}=(V_i,f_\alpha)$, $i\in Q_0$, $\alpha \in Q_1$, where $V_i$ are $\FF_1$-vector spaces and $f_\alpha:V_{s(\alpha)} \to V_{t(\alpha)}$ are $\FF_1$-linear maps. 
\end{mydef} 

\begin{mydef}\cite[Definition 4.3]{szczesny2011representations}
	Let $\mathbb{V}=(V_i,f_\alpha)$ be an $\FF_1$-representation of $Q$. 
	\begin{enumerate}
		\item 
		The \emph{dimension} of $\mathbb{V}$ is $\dim(\mathbb{V})=\sum_{i \in Q_0} \dim(V_i)$.
		\item 
		The \emph{dimension vector} of $\mathbb{V}$ is the $|Q_0|$-tuple $\underline{\dim}(\mathbb{V})=(\dim(V_i))_{i \in Q_0}$.
	\end{enumerate}
The notions of nilpotent $\FF_1$-representations and indecomposable $\FF_1$-representations are defined similar to the case over fields. See \cite[Section 2]{jun2021coefficient} for precise definitions. 
\end{mydef}

For representations  $\mathbb{V}=(V_i,f_\alpha)$ and $\mathbb{W}=(W_i,g_\alpha)$ of a quiver $Q$ over $\FF_1$, a morphism $\Phi:\mathbb{V} \to \mathbb{W}$ is a collection of $\FF_1$-linear maps $(\phi_i)_{i \in Q_0}$, where $\phi_i:V_i \to W_i$, such that the following commutes for each $i \in Q_0$: 
\begin{equation}
	\begin{tikzcd}[row sep=large, column sep=1.5cm]
		V_{s(\alpha)}\arrow{r}{\phi_{s(\alpha)}}\arrow{d}{f_\alpha}
		& W_{s(\alpha)} \arrow{d}{g_\alpha} \\
		V_{t(\alpha)} \arrow{r}{\phi_{t(\alpha)}} 
		& W_{t(\alpha)}
	\end{tikzcd}
\end{equation}
We let $\textrm{Rep}(Q,\mathbb{F}_1)$ be the category of representations of $Q$ over $\mathbb{F}_1$, and  $\textrm{Rep}(Q,\mathbb{F}_1)_{\textrm{nil}}$ be the full subcategory consisting of nilpotent representations.  

For any field $k$ be a field and $\mathbb{F}_1$-vector space $V$, we let $V_k$ be the vector
space whose basis is $V- \{0_V\}$. For an $\mathbb{F}_1$-linear map $f : V \to W$, we let $f_k$ be the linear
from $V_k$ to $W_k$ induced from $f$. This defines a \emph{base change functor} $k\otimes_{\mathbb{F}_1}- : \textrm{Rep}(Q,\mathbb{F}_1) \rightarrow \textrm{Rep}(Q,k)$. For an $\mathbb{F}_1$-representation $M$ of $Q$, we write $M_k := k\otimes_{\mathbb{F}_1}M$. Of course, the base change functor restricts to nilpotent representations as well.

\subsection{Coefficient quivers}
Coefficients quivers were first introduced by Ringel \cite{ringel1998exceptional} as a combinatorial gadget to study representations of quivers. In \cite{jun2020quiver, jun2021coefficient}, the first and third authors recast $\FF_1$-representations in terms of other quivers (also called coefficient quivers).

\begin{mydef}\label{definition: coefficient quiver}
Let $\mathbb{V}$ be an $\FF_1$-representation of $Q$. The coefficient quiver $\Gamma_{\mathbb{V}}$ has the following vertex set:
\[
	(\Gamma_{\mathbb{V}})_0 = \bigsqcup_{v \in Q_0}{(\mathbb{V}_v\setminus\{0\})}.
	\]
		For each $\alpha \in Q_1$, draw an arrow $(\alpha,i,j)$ in $\Gamma_{\mathbb{V}}$ from $i$ to $j$  if and only if $\mathbb{V}_{\alpha}(i) = j$. Then associate to $\Gamma_{\mathbb{V}}$ a quiver map $c_{\mathbb{V}} : \Gamma_{\mathbb{V}} \rightarrow Q$. This map is defined on vertices via the formula
	\[ 
	c_{\mathbb{V}}(i) = v,
	\]
	for all $i \in \mathbb{V}_v$. It is defined on arrows via the formula 
	\[ 
	c_{\mathbb{V}}(\alpha,i,j) = \alpha.
	\]
\end{mydef}     

In \cite[Proposition 3.7]{jun2021coefficient}, it is proven that $\Rep(Q,\FF_1)$ is equivalent to a category $\mathcal{C}_Q$ whose objects are winding maps $c : \Gamma \rightarrow Q$. Below we recall the construction, along with several related results which were originally proved in \cite{jun2020quiver}.

\begin{mydef} 
A \emph{winding} is a quiver map $c : \Gamma \rightarrow Q$ such that for all $\alpha, \beta \in \Gamma_1$ with $\alpha \neq \beta$, $c(\alpha) = c(\beta)$ implies $s(\alpha) \neq s(\beta)$ and $t(\alpha) \neq t(\beta)$. 
\end{mydef}

Note that if $M$ is an $\mathbb{F}_1$-representation of $Q$, then $c_M : \Gamma_M \rightarrow Q$ is a winding. In fact, the converse holds as well: a winding $c : \Gamma \rightarrow Q$ determines an $\mathbb{F}_1$-representation of $Q$ that is unique up to isomorphism.

\begin{mydef}\cite[Definition 3.8]{jun2020quiver} 
Let $c : \Gamma \rightarrow Q$ and $c' : \Gamma' \rightarrow Q$ be quiver maps.
\begin{enumerate}
	\item 
A \emph{coefficient morphism} $c \rightarrow c'$ is a quiver map $\phi : \Gamma \rightarrow \Gamma'$ such that the following diagram commutes:
\begin{equation*}
\begin{tikzcd}[row sep=2em]
\Gamma \arrow[rr,"\phi"] \arrow[dr,swap,"c"]
&& \Gamma' \arrow[dl,,"c'"] \\
& Q
\end{tikzcd}.
\end{equation*}
\noindent A coefficient morphism $\phi : c \rightarrow c'$ is a \emph{coefficient isomorphism} if and only if it is bijective on vertices and arrows.
\item 
A collection of vertices $\mathcal{U}$ in $\Gamma$ is said to be \emph{successor-closed} if for every oriented path from $v$ to $u$ in $\Gamma$, $u \in \mathcal{U}$ implies $v \in \mathcal{U}$ as well. A full subquiver is said to be successor-closed if its vertex set is successor-closed. 
\item 
A collection of vertices $\mathcal{D}$ in $\Gamma$ is said to be \emph{predecessor-closed} if for every oriented path from $d$ to $v$ in $\Gamma$, $d \in \mathcal{D}$ implies $v \in \mathcal{D}$ as well. A full subquiver is said to be predecessor-closed if its vertex set is predecessor-closed.
\end{enumerate}
\end{mydef}

\begin{pro}\cite[Lemma 3.10]{jun2020quiver} 
Let $Q$ be a quiver. For any two $\FF_1$-representations $\mathbb{V}= (V_u, f_{\alpha})$ and $\mathbb{W} = (W_u, g_{\alpha})$ of $Q$, there is a bijection between $\Hom(\mathbb{V},\mathbb{W})$ and the set of coefficient isomorphisms from successor-closed full subquivers of $\Gamma_\mathbb{V}$ to predecessor-closed full subquivers of $\Gamma_\mathbb{W}$.
\end{pro}

\begin{cor} 
Let $Q$ be a quiver, with $\mathbb{F}_1$-representations $M$ and $N$ of $Q$. Then $M\cong N$ if and only if there is a coefficient isomorphism between $\Gamma_M$ and $\Gamma_N$. 
\end{cor}

\begin{mydef}\cite{jun2021coefficient}
Let $Q$ be a quiver. Define $\mathcal{C}_Q$ to be the category whose objects are windings of quivers $F:S \to Q$. A morphism $\phi : (S,F)\rightarrow (S',F')$ is an ordered triple $\phi = (\mathcal{U}_{\phi}, \mathcal{D}_{\phi}, c_{\phi})$, where 
\begin{enumerate} 
\item $\mathcal{U}_{\phi}$ is a full subquiver of $S$ whose vertex set is successor closed,
\item $\mathcal{D}_{\phi}$ is a full subquiver of $S'$ whose vertex set is predecessor closed, and
\item $c_{\phi} : \mathcal{U}_{\phi} \rightarrow \mathcal{D}_{\phi}$ is a coefficient isomorphism.  
\end{enumerate}
If $(S,F) \xrightarrow[]{\phi} (S',F')$ and $(S',F') \xrightarrow[]{\psi} (S'', F'')$ are two morphisms in $\mathcal{C}_Q$, their composition $(S,F)\xrightarrow[]{\psi\circ\phi} (S'',F'')$ is the ordered triple
\begin{equation}\label{eq: composition} 
\psi \circ \phi =  \left( \mathcal{U}_{\psi\circ \phi}, \mathcal{D}_{\psi \circ \phi}, c_{\psi\circ\phi} \right) = \left( c_{\phi}^{-1}\left( \mathcal{U}_{\psi} \cap \mathcal{D}_{\phi} \right), c_{\psi}\left( \mathcal{U}_{\psi}\cap \mathcal{D}_{\phi} \right), c_{\psi} \circ c_{\phi} \right).
\end{equation}   
\end{mydef} 

\begin{pro}\cite[Proposition 3.7]{jun2021coefficient} 
Let $Q$ be a quiver. Then there is an equivalence of categories $F : \mathcal{C}_Q \rightarrow \operatorname{Rep}(Q,\FF_1)$ that restricts to an equivalence between $\operatorname{Rep}(Q,\FF_1)_{\nil}$ and the full subcategory of $\mathcal{C}_Q$ whose objects are windings $c : \Gamma \rightarrow Q$ with $\Gamma$ acyclic.
\end{pro}

We end this section with a lemma that will be useful in Section \ref{section: Lie algebra}.

\begin{lem}\cite[Lemma 3.12]{jun2020quiver}\label{lemma: lemma from p1}
	Let $Q$ be a quiver, $E$ an $\FF_1$-representation of $Q$, and $X$ be a subrepresentation of $E$.
	\begin{enumerate}
		\item 
		$\Gamma_X$ is the full subquiver of $\Gamma_E$ obtained by removing vertices (and arrows with adjacent to those vertices) which do not correspond to $X$. The map $c_{X}$ is the restriction of $c_{E}$ to this subquiver.
		\item 
		The coefficient quiver $\Gamma_{E/X}$ of the quotient $E/X$ is the full subquiver of $\Gamma_{E}$ obtained by removing vertices (and arrows adjacent to those vertices) corresponding to $X$. The map $c_{E/X}$ is the restriction of $c_E$ to this subquiver. 
		\item 
		If
		\[ 
		0 \rightarrow X \xrightarrow[]{f} E \xrightarrow[]{g} Y \rightarrow 0 
		\] 
		is a short exact sequence of $\FF_1$-representations of $Q$, then $\Gamma_E$ is obtained from the disjoint union $\Gamma_X \sqcup \Gamma_Y$ by adding certain $\alpha$-colored arrows from $\alpha$-sinks of $\Gamma_Y$ to $\alpha$-sources of $\Gamma_X$, for each $\alpha \in Q_1$. Under this decomposition, $\Gamma_X$ (resp.~$\Gamma_Y$) is a predecessor-closed (resp.~successor-closed) subquiver of $\Gamma_E$.\footnote{By a predecessor-closed (resp.~successor-closed) subquiver of $\Gamma_E$, we mean a full subquiver $\Gamma$ such that for $u,v \in V(\Gamma_E)$ and any oriented path from $v$ to $u$ if $u \in V(\Gamma)$ (resp.~$v \in V(\Gamma)$), then $v \in V(\Gamma)$ (resp.~$u \in V(\Gamma)$). }
	\end{enumerate}		
\end{lem}




\subsection{Growth of nilpotent, indecomposable $\FF_1$-representation}

In this subsection, we briefly review main definitions and theorems on the growth function associated to a quiver $Q$, introduced in \cite{jun2020quiver}.

 \begin{mydef} \label{definition: indecomposables growth}
Let $Q$ be a quiver. Then $\NI_Q : \NN \rightarrow \NN$ is the function such that
\begin{equation}
		\NI_Q(n) = \#\{\text{isomorphism classes of $n$-dimensional indecomposables in $\Rep(Q,\FF_1)_{\nil}$} \}. 
\end{equation}
\end{mydef} 	

\begin{mydef}\label{definition: order, equivalence relation}
For quivers $Q$, $Q'$, we define a preorder $\leq_{\textrm{nil}}$ as follows:
\begin{equation}\label{eq: order}
Q \leq_{\textrm{nil}} Q' \iff \exists~x \in \mathbb{R}_{>0},~~y \in \mathbb{N} \textrm{ such that } 
\NI_Q(n) \leq x\NI_{Q'}(yn) \quad \forall n \gg 0.
\end{equation}
The preorder \eqref{eq: order} induces an equivalence relation $\approx_{nil}$ as follows:
\begin{equation}
Q \approx_{nil} Q' \iff Q	\leq_{\textrm{nil}} Q' \textrm{ and } Q' \leq_{\textrm{nil}} Q.
\end{equation}
\end{mydef} 

\begin{mydef}\label{d.bigO}
Given two functions $f: \NN \rightarrow \mathbb{R}_{>0}$ and $g : \NN \rightarrow \mathbb{R}_{>0}$, we write $f = O(g)$ or $f(n) = O(g(n))$ if there exists a positive real number $M > 0$ such that $f(n) \le Mg(n)$ for $n\gg 0$. 
\end{mydef} 
\begin{rmk}\label{r.functions}
The definition above implies that $Q \le_{\textrm{nil}} Q'$ if and only if $\NI_Q(n) = O(\NI_{Q'}(yn))$ for some $y \in \NN$. We can then extend the definition of $\le_{\textrm{nil}}$ to all functions $\NN \rightarrow \mathbb{R}$ by mandating $f \le_{\textrm{nil}} g$ if and only if $f(n) = O(g(yn))$ for some $y \in \NN$. Then $f \approx_{\textrm{nil}} g$ if and only if $f \le_{\textrm{nil}} g$ and $g \le_{\textrm{nil}}f$. It will be useful throughout to keep the following elementary fact in mind: $a^n \approx_{\nil} b^n$ for any two real numbers $a, b \in (1,\infty )$.  
\end{rmk}

Recall that a quiver $Q$ has finite representation type
over $\FF_1$ if and only if $\textrm{Rep}(Q,\FF_1)_{nil}$ has finitely many isomorphism classes of indecomposables. In terms of the order relation as in Definition \ref{definition: order, equivalence relation}, we have the following:
	
\begin{mythm}\cite[Theorem 5.3]{jun2020quiver}
Let $Q$ be a connected quiver. The following are equivalent.
\begin{enumerate}
	\item 
$Q$ is of finite representation type over $\FF_1$. 
\item 
$Q$ is a tree.
\item 
$Q \approx_{nil} \wild_0$. 
\end{enumerate}
\end{mythm}

Recall that a quiver $Q$ is of bounded representation type over $\FF_1$ if there exists a positive number $M$ such that $\NI_Q(n) \leq M$ for $n\gg 0$. We have the following. 

\begin{mythm}\cite[Theorem 5.14]{jun2020quiver}
	Let $Q$ be a connected quiver. The following are equivalent.  
\begin{enumerate} 
\item $Q$ is of bounded type over $\FF_1$.
\item $Q \le_{\nil} \wild_1$.
\item $Q$ is either a tree or a cycle quiver.
\end{enumerate}
Moreover, $Q$ is a cycle quiver if and only if $Q \approx_{nil} \wild_1$. 
\end{mythm}

Finally, we have the following showing that $\wild_2$ is a ``universal bound.''

\begin{mythm}\cite[Theorem 4.6, Proposition 5.4]{jun2020quiver}
For any connected quiver $Q$, we have $Q \leq_{\nil} \wild_2$. Furthermore, $Q\approx_{\nil} \wild_2$ if $\operatorname{rank}H_1(Q,\mathbb{Z}) \geq 2$.
\end{mythm}


\subsection{Hall algebras and Lie algebras of $\FF_1$-representations} \label{subsection: Hall algebras}

In this subsection, we briefly recall the Hall algebras of quiver representations over $\FF_1$. For details, we refer the reader to \cite{jun2020quiver, jun2021coefficient, szczesny2011representations}. 

Let $\textrm{Iso}(Q)$ be the set of isomorphism classes of $\FF_1$-representations of $Q$. The underlying set of is the following:
\begin{equation}\label{eq: hall algebra}
	H_Q:=\{f:\textrm{Iso}(Q) \to \mathbb{C} \mid f([M])=0\textrm{ for all but finitely many $[M]$.}\}
\end{equation}
The \emph{support} of an element $f \in H_Q$ is the set $\operatorname{supp}(f) = \{ [M] \in \textrm{Iso}(Q)\mid f([M]) \neq 0 \}$. For each $[M] \in \textrm{Iso}(Q)$, let $\delta_{[M]}$ be the delta function in $H_Q$ supported at $[M]$. In particular, we consider $H_Q$ as the vector space spanned by $\{\delta_{[M]}\}_{[M] \in \textrm{Iso}(Q)}$ over $\mathbb{C}$. To ease the notation, we will denote the delta function $\delta_{[M]}$ by $[M]$. One defines the following multiplication on the elements in $\textrm{Iso}(Q)$:
\begin{equation}\label{eq: hall product}
	[M]\cdot[N]:=\sum_{R \in \textrm{Iso}(Q)} \frac{a^R_{M,N}}{a_Ma_N}[R],
\end{equation}
where $a_M=|\textrm{Aut}(M)|$ and $a^R_{M,N}$ is the number of ``short exact sequences'' of the form:\footnote{For a given morphism $\varphi$ between $\FF_1$-vector spaces, $\ker(\varphi)$ and $\textrm{coker}(\varphi)$ are defined analogous to the classical case. By a short exact sequence we mean that $\ker=\textrm{coker}$. See \cite{jun2020quiver} for details.}
\[
0 \to N \to R \to M \to 0.
\]
Then, one can easily check the following equality as in the classical case:
\[
\frac{a^R_{M,N}}{a_Ma_N}=|\{L \leq R \mid L\simeq N\textrm{ and } R/L \simeq M\}|.
\]
By linearly extending the multiplication \eqref{eq: hall product} to $H_Q$, we obtain an associative algebra $H_Q$ over $\mathbb{C}$. Moreover, one may check that $H_Q$ is also equipped with the coproduct defined as follows:
\begin{equation}\label{eq: hall coprod}
	\Delta:H_Q\to H_Q\otimes_\mathbb{C}H_Q, \quad \Delta(f)([M],[N])=f([M\oplus N]).
\end{equation}

The Hall algebra $H_Q$ is graded, connected, cocommutative Hopf algebras, and hence by Milnor-Moore Theorem, $H_Q$ is isomorphic to the universal enveloping algebras of its Lie subalgebras $\mathfrak{n}_Q$ of primitive elements, i.e., indecomposable representations. In particular, one can describe the algebraic structure of $H_Q$ in terms of certain ``stacking operations'' of coefficient quivers; See \cite{jun2020quiver} for details. Note that for the Hall algebra $H_{Q,\nil}$ of nilpotent representations, the same statements hold. In particular, $H_{Q,\nil}$ is the universal enveloping algebra of the Lie algebra $\mathfrak{n}_{Q,\nil}$ of nilpotent indecomposable representations.

\subsection{Gradings and nice sequences}

We will recall some basic definitions from \cite{jun2021coefficient}. By a grading on a quiver $\Gamma$, we simply mean a function $\partial:\Gamma_0 \to \mathbb{Z}$. We will be interested in a finite sequence of gradings on $\Gamma$ satisfying some conditions. The existence of such a sequence of gradings allows one to compute the Euler characteristic of certain quiver Grassmannians in a purely combinatorial way. 

\begin{mydef}\label{definition: winding}
Let $c:\Gamma \to Q$ be a winding.
\begin{enumerate}
\item 
 A nice grading is a grading $\partial:\Gamma_0 \to \mathbb{Z}$ such that for $\alpha,\beta \in \Gamma_1$ if $c(\alpha)=c(\beta)$, then 
 \begin{equation}\label{equation: winding}
 \partial(s(\alpha))-\partial(t(\alpha)) = \partial(s(\beta))-\partial(t(\beta))
 \end{equation}
\item 
A nice grading $\partial$ is said to be \emph{non-degenerate} if for each $\alpha \in Q_1$ and $\beta \in \Gamma_1$ such that $c_1(\beta)=\alpha$, one has $\partial (t(\beta)) - \partial (s(\beta)) \neq 0$. 
\item 
Let $\{\partial_0,\dots,\partial_i\}$ be gradings on $\Gamma$. A grading $\partial$ is $\{\partial_0,\dots,\partial_k\}$-nice if for each $\alpha,\beta \in \Gamma_1$, the following conditions:
\begin{equation}\label{eq: condition}
 c(\alpha)=c(\beta)\quad\textrm{and}\quad  \partial_i(s(\alpha))=\partial_i(s(\beta)), \quad \partial_i(t(\alpha))=\partial_i(t(\beta)) \quad \forall i=1,\dots k.
\end{equation}
imply that 
\begin{equation}
 \partial(s(\alpha))-\partial(t(\alpha)) = \partial(s(\beta))-\partial(t(\beta))
\end{equation}
\item 
A sequence $\{ \partial\} := \{\partial_i\}_{i \in \mathbb{N}}$ is a nice sequence if for each $i$, $\partial_i$ is $(\partial_1,\dots,\partial_{i-1})$-nice and $\partial_1$ is a nice grading. 
\end{enumerate}
\end{mydef}

\begin{mydef}\label{definition: disting}
Let $c:\Gamma \to Q$ be a winding and $\{\partial\}$ be a nice sequence on $\Gamma$. $\{\partial\}$ is said to distinguish the vertices of $\Gamma$ if for each $u,v \in \Gamma_0$ there exists $\partial_i$ such that $\partial_i(u) \neq \partial_i(v)$. In particular, a nice grading $\partial$ distinguishes vertices of $\Gamma$ if for each $u \neq v \in \Gamma_0$, $\partial(u) \neq \partial (v)$. 
\end{mydef}

\begin{mythm}\cite[Proposition 5.2]{jun2021coefficient}\label{theorem: quiver grass thm}
Let $Q$ be a quiver and $M$ be an $\FF_1$-representation of $Q$; let $c:\Gamma \to Q$ be the corresponding winding. Suppose that $\Gamma$ has a nice sequence distinguishing the vertices of $\Gamma$. Then for any dimension vector $\underline{d} \leq \underline{\dim}(M)$, the Euler characteristic $\chi_{\underline{d}}(M_\mathbb{C})$ of the quiver Grassmannian $\textrm{Gr}^Q_{\underline{d}}(M_\mathbb{C})$ is the cardinality of following set:
\begin{equation}
\{N \leq M \mid \underline{\dim}(N)=\underline{d}\}
\end{equation}
\end{mythm}

The following example shows that our results are more general than Haupt \cite{haupt2012euler}.

\begin{myeg}\cite[Example 3.7]{haupt2012euler} \label{example: not band}
	Consider the following quiver:
	\[
	Q=\left(
	\begin{tikzcd}
		1 \arrow[r,"\alpha"] & 2 \arrow[r,bend right = 60,swap, "\beta_4"] \arrow[r,shift right,swap,"\beta_3"]\arrow[r,shift left,"\beta_2"] \arrow[r,bend left=60, "\beta_1"]& 3
	\end{tikzcd}
	\right)
	\]
	Consider the following coefficient quiver $\Gamma_M$:
	\[
	F:\Gamma_M=\left(
	\begin{tikzcd}
		& &  & 3 \\
		1\arrow[r,"\alpha"]	&2   \arrow[rru,"\beta_1"] \arrow[rrd,swap,"\beta_4"]&2' \arrow[ru,swap,"\beta_2"]  \arrow[rd,"\beta_3"]&\\
		&&& 3'	
	\end{tikzcd}
	\right)
	\to Q=\left(
	\begin{tikzcd}
		1 \arrow[r,"\alpha"] & 2 \arrow[r,bend right = 60,swap, "\beta_4"] \arrow[r,shift right,swap,"\beta_3"]\arrow[r,shift left,"\beta_2"] \arrow[r,bend left=60, "\beta_1"]& 3
	\end{tikzcd}
	\right)
	\]
	Then, $M_\mathbb{C}$ is not a tree, band, nor thin. But, notice that $\Gamma_M$ is a pseudotree and hence we can apply theorems in \cite{jun2021coefficient} to compute the following:
	\[
	\chi_{(0,1,2)}^Q(M_\mathbb{C})=2
	\]
	since there are two predecessor closed subquivers of $\Gamma_M$ corresponding to $(0,1,2)$ as follows:
	\[
	\begin{tikzcd}
		3 & 2 \arrow[l,swap, "\beta_1"] \arrow[r,"\beta_4"] & 3'
	\end{tikzcd}, \quad \begin{tikzcd}
		3 & 2' \arrow[l,swap,"\beta_2"] \arrow[r,"\beta_3"] & 3'
	\end{tikzcd}
	\]
\end{myeg}

\section{Growth of $\FF_1$-representations} \label{section: growth}

In this section, we prove Theorem A. In Subsection \ref{subsection: one loop}, we first investigate a special case of rooted trees with one loop, and then use this to study the general case of pseudotrees in subsequent Subsections \ref{subsection: equiori} (equioriented case) and \ref{subsection: reversing arrows} (acyclic case). 

\subsection{Rooted trees with one loop}\label{subsection: one loop}

Let $T$ be a rooted tree, and $Q_T$ be the quiver obtained by identifying the vertex of $\wild_1$ and the root of $T$ as follows.
\begin{equation}\label{eq: rooted tree and a loop}
	Q_T:=\begin{tikzcd}
		T \arrow[d, no head,dashed]\\
		\bullet
		\arrow[out=240,in=300,loop,swap]
	\end{tikzcd}
\end{equation}
Let $S_k$ be the set of rooted subtrees of $T$ consisting of exactly $k$ vertices sharing the same root, and $s_k=|S_k|$. Coefficient quivers of $n$-dimensional indecomposable representations of $Q_T$ are of the following form:
\begin{equation}\label{eq: tree}
	\begin{tikzcd}[arrow style=tikz,>=stealth,row sep=2em]
		& T_i \arrow[d, no head, dashed] &  & T_\ell \arrow[d,no head, dashed] & & T_j \arrow[d,no head, dashed]  \\
		  & \bullet   \arrow[r] & \cdots  \arrow[r] & \bullet\arrow[r] & \bullet \arrow[r] & \bullet &
	\end{tikzcd}
\end{equation}
where $T_i$ is a rooted subtree of $T$ with exactly $i$ vertices, and the number of total vertices appearing in \eqref{eq: tree} is precisely $n$. Note that the root of each $T_i$ is identified with a vertex of the oriented path. In fact, this can be equivalently thought as follows: we consider a partition of the set $[n]=\{1\ldots , n\}$ as 
\[
[n] = \lambda_1 \sqcup \cdots \sqcup \lambda_r, 
\]
 where there exists a sequence $1 = n_1 \le n_2 \le \ldots \le n_r \le n_{r+1} = n$ such that $\lambda_i = \{ n_i, n_i+1,\ldots , n_{i+1} \}$ for all $i$. We first build a monochromatic oriented path $\lambda_1 \xrightarrow[]{\alpha} \lambda_2 \xrightarrow[]{\alpha}\cdots \xrightarrow[]{\alpha} \lambda_r$, where $\alpha$ is the loop in $\wild_1$. If $|\lambda_i| =a_i$, we then attach a rooted subtree $T_{a_i}$ whose number of vertices is exactly $a_i$ at $\lambda_i$ (including the vertex of the oriented path which is identified with the root of $T_i$). 
\begin{myeg}
Consider the following quivers.
\begin{equation}
	Q_T=\begin{tikzcd}[column sep =0.7em]
	\bullet \arrow[dr]	&	 & \bullet \arrow[dl] \\
&	\bullet \arrow[d] & \\\
	&	\bullet \arrow[out=240,in=300,loop,swap] &
	\end{tikzcd}, \qquad T= \begin{tikzcd}[column sep =0.7em]
	\bullet \arrow[dr]	&	 & \bullet \arrow[dl] \\
	&	\bullet \arrow[d] & \\
	&	\bullet &
\end{tikzcd}
\end{equation}
Then, we have
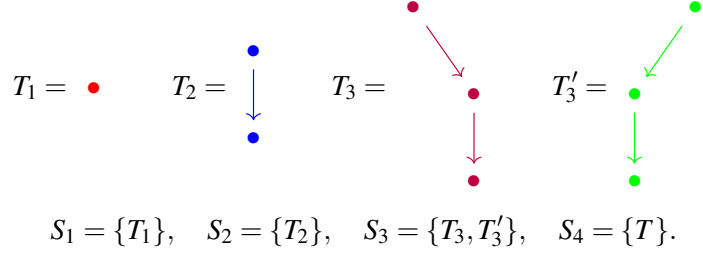
\[
T_1=  \begin{tikzcd}
	\textcolor{red}{\bullet} 
\end{tikzcd} \qquad T_2= \begin{tikzcd}
		\textcolor{blue}{\bullet} \arrow[d,blue]  \\
		\textcolor{blue}{\bullet}
\end{tikzcd} \qquad T_3=\begin{tikzcd}[column sep =0.7em]
\textcolor{purple}{\bullet}\arrow[dr,purple]	&	 \\
&	\textcolor{purple}{\bullet} \arrow[d,purple]  \\
&	\textcolor{purple}{\bullet}
\end{tikzcd} \qquad T_3'=\begin{tikzcd}[column sep =0.7em]
		 & \textcolor{green}{\bullet} \arrow[dl,green] \\
	\textcolor{green}{\bullet} \arrow[d,green] & \\
\textcolor{green}{\bullet} &
\end{tikzcd}
\]
\[
S_1=\{T_1\}, \quad S_2=\{T_2\}, \quad S_3=\{T_3,T_3'\}, \quad S_4=\{T\}. 
\]
Let $\mathbb{V}$ be a $12$-dimensional representation of $Q_T$ whose coefficient quiver is as follows:
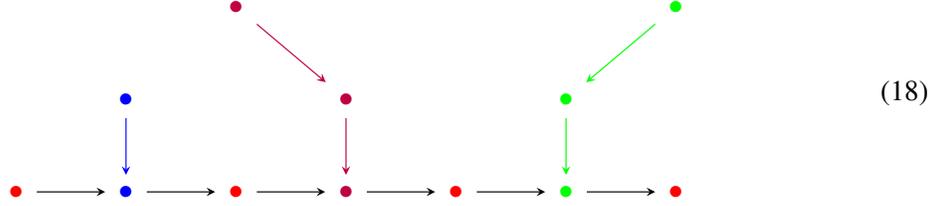
\begin{equation}\label{eq: tree example}
	\begin{tikzcd}[arrow style=tikz,>=stealth,row sep=2em]
			&  & \textcolor{purple}{\bullet} \arrow[dr,purple] & & &  & \textcolor{green}{\bullet} \arrow[dl,green] \\	
		& \textcolor{blue}{\bullet} \arrow[d,blue] &  & \textcolor{purple}{\bullet} \arrow[d, purple] & & \textcolor{green}{\bullet} \arrow[d,green]  \\
		\textcolor{red}{\bullet}  \arrow[r]  & \textcolor{blue}{\bullet} \arrow[r] & \textcolor{red}{\bullet}  \arrow[r] & \textcolor{purple}{\bullet}\arrow[r] & \textcolor{red}{\bullet}  \arrow[r] & \textcolor{green}{\bullet} \arrow[r]& \textcolor{red}{\bullet} 
	\end{tikzcd}
\end{equation}
This can be identified with the following partition:
\[
\lambda_1={1}, \quad \lambda_2=\{2,3\}, \quad \lambda_3=\{4\}, \quad \lambda_4=\{5,6,7\},\quad  \lambda_5=\{8\}, \quad \lambda_6=\{9,10,11\},\quad  \lambda_7=\{12\},
\]
where we attach $T_1$ to each red vertices \{$\lambda_1,\lambda_3,\lambda_7$\}, $T_2$ to $\lambda_2$, $T_3$ to $\lambda_4$, and $T_3'$ to $\lambda_6$. 
\end{myeg}

Fix a rooted tree $T$ and set $Q = Q_T$. One may observe that $\NI_Q$ can be written recursively (by considering the size $a_r$ of the last partition $\lambda_r$ of $\lambda=\lambda_1,\dots,\lambda_r$) as follows:

\begin{equation}\label{eq: recursion}
\NI_Q(n)=s_1\NI_Q(n-1)+s_2\NI_Q(n-2)+s_3\NI_Q(n-3)+ \cdots + s_{t}\NI_Q(n-t), \quad \forall n \geq 2t,
\end{equation}
where $t$ is the number of vertices in a tree $T$. More explicitly, consider an $n$-dimensional nilpotent indecomposable representation $M$. Let $\lambda = \lambda_1 \sqcup \cdots \sqcup \lambda_r$ be the associated partition, with $T_i$ the rooted subtree corresponding to $\lambda_i$. Since $n \geq 2t$, the representation $M'$ obtained from $M$ by deleting $T_r$ is indecomposable and satisfies $\dim_{\FF_1}(M') \geq t$. In particular, the coefficient quiver $\Gamma_{M'}$ will have at least one vertex colored with the root of $T$. Clearly, $M$ can be reconstructed from the ordered pair $(M',T_r)$ by inserting an arrow whose source is in $M'$ and whose target is in $T_r$. Under the assumption $n \geq 2t$, there are $s_{a_r}\NI_Q(n-a_r)$ choices for such an ordered pair.\footnote{If $n< 2t$, then deletion will generally produce fewer than $s_{a_r}\NI_Q(n-a_r)$ ordered pairs $(M',T_r)$. The issue is that $M'$ always contains a root-colored vertex, whereas there are generally indecomposable representations without root-colored vertices in dimension $<t$.}

From the recursion formula \eqref{eq: recursion}, we obtain the following polynomial to compute $\NI_Q(n)$:
\begin{equation} \label{eq: recursion polynomial}
	p_T(x):= x^{t}-s_1x^{t-1}- s_{2}x^{t-2} - \cdots -s_t=0. 
\end{equation} 
Let $c_1,\ldots , c_r$ be the distinct roots of $p_T(x)$. Then the theory of linear recurrences implies that we may write 
\[ 
\NI_Q(n) = \sum_{i=1}^r{f_i(n)c_i^n}
\] 
for suitable polynomials $f_i(n)$ and $n \geq 2t$. Let $c$ be a root of \eqref{eq: recursion polynomial} with maximal absolute value. Then $\NI_Q(n)=O(n^t|c|^n)$, and in particular $\NI_Q(n) = O(|c|^{2n})$ if $|c|>1$. If $T$ has more than one vertex, then this condition holds. Indeed, if $t>1$ then $\operatorname{deg}(p_T)>1$ and $s_t = 1 = s_1$. In particular,
\[
p_T(1) = 1 - \sum_{i=1}^t{s_i} \le -1,
\]
and $p_T(x)\rightarrow +\infty$ as $x\rightarrow +\infty$. Hence, $p_T$ has a root $c$ with absolute value strictly greater than $1$. In particular, $\NI_Q(n) = O(a^n)$ for some $a \in (1, \infty )$.


Now, we have the following. We will use the following notation:
\[
Q_{\bullet \rightarrow \bullet}=\begin{tikzcd}
	\bullet \\
	\bullet \arrow[u]
	\arrow[out=240,in=300,loop]
\end{tikzcd}, \quad Q_{\bullet \leftarrow \bullet}=\begin{tikzcd}
\bullet \arrow[d]\\
\bullet 
\arrow[out=240,in=300,loop]
\end{tikzcd}
\]

\begin{pro}\label{proposition: speicla case}
Let $T$ and $T'$ be rooted trees, and suppose that $T$ has more than one vertex. Then $Q_T \approx_{nil} Q_{T'}$ if and only if $T'$ has more than one vertex.
\end{pro}
\begin{proof}
The necessity follows from the fact that $Q_{T'}$ has bounded representation type if and only if $T'$ has one vertex. Since $T$ has more than one vertex, either $Q = Q_{\bullet \rightarrow \bullet}$ or $Q= Q_{\bullet \leftarrow \bullet}$ is a subquiver of $Q_T$. Note that for either choice of $Q$, $\NI_Q(n) = \NI_Q(n-1) + \NI_Q(n-2)$ for $n \geq 4$. Furthermore, $\NI_Q(4) = 5$ and $\NI_Q(5) = 8$, so Binet's Formula implies that
\[ 
\NI_Q(n) = \frac{1}{\sqrt{5}}\left[ \left( \frac{1+\sqrt{5}}{2} \right)^{n+1} - \left(\frac{1-\sqrt{5}}{2} \right)^{n+1} \right]
\] 
for all $n\geq 4$. In particular, $\NI_Q(n)$ is asymptotic to $\frac{1}{\sqrt{5}}\left( \frac{1+\sqrt{5}}{2} \right)^{n+1}$, and so $\NI_Q(n) > d^n$ for large enough $n$, where $d = \sqrt{5}/2$. 
Since $Q$ is a subquiver of $Q_T$, $Q \leq_{\textrm{nil}}Q_T$. We claim that also $Q_T\leq_{\textrm{nil}}Q$. By the argument in the preceding paragraph, there exists an $a \in (1,\infty)$ such that $\NI_{Q_T}(n) = O(a^n)$.             
But note that 
\[ 
a^n =  d^{\log_d(a)n} \le d^{\lceil \log_d(a)\rceil n}
\] 
and hence 
\[ 
\NI_{Q_T}(n) = O(a^n) = O(d^{\lceil \log_d(a) \rceil n}) = O(\NI_Q(K n)),
\] 
where $K := \lceil \log_d(a)\rceil$. This shows $\NI_{Q_T} \leq_{\textrm{nil}} Q$, as claimed. It follows that $Q_T \approx_{\textrm{nil}} Q$, where either  $Q = Q_{\bullet \rightarrow \bullet}$ or $Q= Q_{\bullet \leftarrow \bullet}$. Since clearly  $ Q_{\bullet \rightarrow \bullet} \approx_{\textrm{nil}}Q_{\bullet \leftarrow \bullet}$, the result follows.
\end{proof}

\begin{myeg}\label{example: one arrow and two arrows}
	Consider the following two quivers. 
	\begin{equation}\label{quiverquiver}
		Q:=\begin{tikzcd}
			\bullet \arrow[d]\\
			\bullet 
			\arrow[out=240,in=300,loop]
		\end{tikzcd}, \quad Q':=\begin{tikzcd}
			\bullet \arrow[d] & \\
			\bullet 	
			\arrow[out=240,in=300,loop] & \bullet\arrow[l]
		\end{tikzcd}
	\end{equation}
We have that $Q \approx_{\textrm{nil}} Q'$. In fact, from $Q$, we obtain the following:
	\[
	\NI_Q(n)=\NI_Q(n-1)+\NI_Q(n-2), \quad x^2-x-1=0, \quad c=\frac{1+\sqrt{5}}{2}. 
	\]
	On the other hand, from $Q'$, we obtain:
	\[
	\NI_{Q'}=\NI_{Q'}(n-1)+2\NI_{Q'}(n-2)+\NI_{Q'}(n-3), \quad x^3-x^2-2x-1=0, \quad d>2. 
	\]
Where $c$ and $d$ are the roots with maximum absolute values for $x^2-x-1$ and $x^3-x^2-2x-1$, respectively. In particular, we have $Q \approx_{nil} Q'$. 
\end{myeg}

\subsection{Pseudotrees with equioriented central cycle} \label{subsection: equiori}

In fact, our argument can be also used for any pseudotree. In this subsection, we first consider pseudotrees with equioriented central cycle:
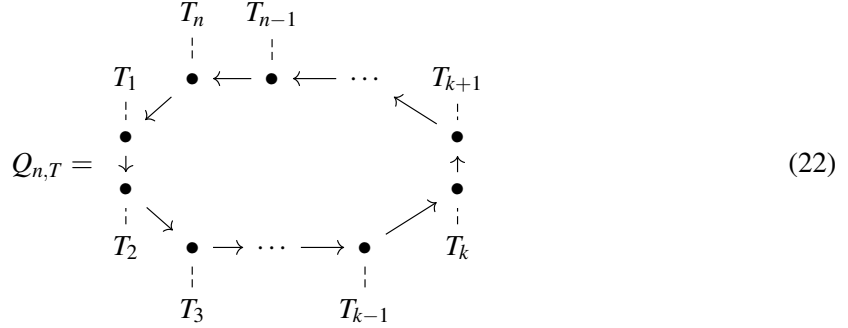
\begin{equation}\label{eq: cycle and trees}
Q_{n,T}=\begin{tikzcd}[row sep =0.6em, column sep =0.5em]
	& T_n \arrow[d,no head, dashed] & T_{n-1} \arrow[d,no head, dashed] &  &  \\
	T_1 \arrow[d,no head, dashed] & \bullet \arrow[dl] & \bullet \arrow[l] & \cdots \arrow[l]& T_{k+1} \arrow[d,no head, dashed] \\
	\bullet \arrow[d] &  &  &  & \bullet \arrow[ul]\\
	\bullet \arrow[dr] &  &  &  & \bullet \arrow[u] \\
	T_2 \arrow[u,no head, dashed] & \bullet \arrow[r] & \cdots  \arrow[r] & \bullet \arrow[ur] & T_{k} \arrow[u,no head, dashed]  \\
	& T_3 \arrow[u,no head, dashed]  & & T_{k-1} \arrow[u,no head, dashed]   \\
\end{tikzcd}   
\end{equation}
where each $T_i$ is a rooted tree and the root is identified with a vertex of the cycle. 

Let $Q=Q_{n,T}$ and $C$ be the central cycle of $Q$. We assume that $C$ is equioriented, so that we may label its vertices clockwise as $v_1$, $v_2$, $\ldots , v_n$ and for each $1 \le i < (n-1)$, there is an arrow $\alpha_i : v_i \rightarrow v_{i+1}$, and an arrow $\alpha_n:v_n \to v_1$. For convenience, we extend the subscripts on the vertices and arrows of $C$ by defining $v_i := v_{i\text{ } \operatorname{mod }n}$ and $\alpha_i := \alpha_{i\text{ } \operatorname{mod} n}$ for all $i \in \mathbb{Z}$. We let $T_i$ denote the rooted tree attached to $v_i$ as in \eqref{eq: cycle and trees}, and let $t_i$ denote the total number of vertices of $T_i$ (including $v_i$ as the root). We also let $s^{(i)}_k$ denote the number of rooted subtrees of $T_i$ with exactly $k$ vertices. The coefficient quiver $\Gamma_M$ of a $d$-dimensional nilpotent indecomposable $\mathbb{F}_1$-representation $M$ of $Q$ is of the following form:\footnote{See \cite[Theorem 5]{szczesny2011representations} and \cite[Theorem 5.12]{jun2020quiver}.}
\begin{equation}\label{eq: treeq}
\Gamma_M=\left(	\begin{tikzcd}[arrow style=tikz,>=stealth,column sep=2em]
		T_i^{(0)} \arrow[d, no head, dashed]	& T_j^{(1)} \arrow[d, no head, dashed] &  & T_k^{(r-3)}  \arrow[d,no head, dashed] & T_p^{(r-2)}   \arrow[d,no head, dashed] &  T_m^{(r-1)}  \arrow[d,no head, dashed]  & T_\ell^{(r)}  \arrow[d,no head, dashed]\\
		v_i \arrow[r]  & v_{i+1}   \arrow[r] & \cdots  \arrow[r] & v_{i+j-3}\arrow[r] & v_{i+j-2} \arrow[r] & v_{i+j-1} \arrow[r]& v_{i+j}
	\end{tikzcd} \right)
\end{equation} 
where $T^{(i)}_d$ is a rooted subtree of a tree $T_i$ with exactly $d$ vertices (the root is again identified with a vertex of the oriented path). The number of vertices appearing in \eqref{eq: treeq} is precisely $d$ since it is associated to a $d$-dimensional representation of $Q_{n,T}$.

We will call the oriented path in \eqref{eq: treeq} the \emph{spine} of $M$, and $j+1$ is called the \emph{spine length}. The vertex $v_i$ is the source (or start) of the spine and $v_{i+j}$ is the target (or finish). We define the following function: 
\begin{mydef}  \label{definition: recursion count}
Let $Q$ be a proper pseudotree, subject to the notational conventions above. For all $1 \le i,f \le n$ and $j, d \in \mathbb{N}$ we define $Q_{i\rightarrow f}(j,d)$ to be the number of isomorphism classes of indecomposable, nilpotent, $d$-dimensional $\mathbb{F}_1$-representations of $Q$ with spine length $j$, starting at $v_i$ and finishing at $v_f$. We also define  
\[
\displaystyle Q_{i\rightarrow f}(d) = \sum_{j\geq 0}{Q_{i\rightarrow f}(j,d)}. 
\]
 Clearly, the following relation holds for $d \geq \max \{t_i\}$: 
\[
\displaystyle \operatorname{NI}_Q(d) = \sum_{1 \le i,f\le n}{Q_{i\rightarrow f}(d)}. 
\] 
\end{mydef}  
We begin with the following fundamental observation:
\begin{pro}  
For $d\geq t_f+t_{f-1}$, we have 
\[ 
\displaystyle Q_{i\rightarrow f}(j,d) = \sum_{k=1}^{t_f}{s^{(f)}_kQ_{i\rightarrow (f-1)}(j-1,d-k)},
\]  
where subscripts referring to vertices of $C$ are taken modulo $n$, as needed. In particular, 
\[ 
\displaystyle Q_{i\rightarrow f}(d) = \sum_{k=1}^{t_f}{s^{(f)}_kQ_{i\rightarrow(f-1)}(d-k)}.
\]
\end{pro} 
\begin{proof} 
Let $M$ be a $d$-dimensional, indecomposable, nilpotent $\mathbb{F}_1$-representation of $Q$ with spine length $j$ starting at $v_i$ and ending at $v_f$. Since $d \geq t_i+t_{i+1}$, we must have $j \geq 2$. Then deleting the last rooted tree in $\Gamma_M$ establishes a bijective correspondence between such $M$ and ordered pairs $(N,T)$, where $N$ is a nilpotent indecomposable with spine length $j-1$ and dimension $d-k$ and $T$ is a rooted subtree of $T_f$ with $k$ vertices (for all $1 \le k \le t_f$). There are $Q_{i\rightarrow (f-1)}(j-1,d-k)$ choices for an indecomposable with spine length $j-1$ and dimension $d-k$, and each can be completed to an indecomposable with spine length $j$ and dimension $d-k$ in exactly $s^{(f)}_k$ ways. The first relation easily follows, and the second follows from the first by summing both sides over all $j\geq 0$. 
\end{proof} 
Fix $i$. Assume that for all $1 \le f \le n$ and $d' < d$, the numbers $Q_{i\rightarrow f}(d')$ are known. Then the above relation implies that we may write $Q_{i\rightarrow f}(d)$ as a non-negative linear combination of the prior values of $Q_{i\rightarrow (f-1)}$;  $Q_{i\rightarrow (f-1)}$ as a non-negative linear combination of prior values of $Q_{i\rightarrow (f-2)}$; etc. In particular, we may write $Q_{i\rightarrow f}(d)$ as a non-negative linear combination of the prior values of $Q_{i\rightarrow (f-n)} = Q_{i\rightarrow f}$. In other words, we have the following: 
\begin{cor} 
For $d \geq t_f+t_{f-1}$, $Q_{i\rightarrow f}(d)$ satisfies a linear recursion with non-negative constant coefficients: 
\[ 
\displaystyle Q_{i\rightarrow f}(d) = \sum_{k=1}^{d_{i\rightarrow f}}{\lambda^{(f)}_kQ_{i\rightarrow f}(d-k)}
\]   
where $\lambda^{(f)}_k \in \mathbb{N}$ and $d_{i\rightarrow f} \in \mathbb{Z}_{>0}$.
\end{cor}  
It is easy to show that $\lambda^{(f)}_{d_{i\rightarrow f}} = 1$. Hence, the characteristic polynomial $p_{i\rightarrow f}(d)$ associated to the above recursion satisfies 
\begin{enumerate} 
\item $p_{i\rightarrow f}(0) = -1.$
\item $\displaystyle\lim_{x\rightarrow \infty}{p_{i\rightarrow f}(x)} = + \infty$.
\end{enumerate}  
It follows that $p_{i\rightarrow f}(d)$ has a root with positive absolute value. In turn, this implies the following: 
\begin{cor} \label{corollary: main cor}
With the above notational conventions, 
\[ 
\displaystyle Q_{i\rightarrow f}(d) = O(r_{i\rightarrow f}^d)
\]  
\[ 
\displaystyle \operatorname{NI}_Q(d) = O(r^d) 
\] 
for suitable $r_{i\rightarrow f}, r \in (1,\infty)$.\footnote{Refer again to the argument above Proposition \ref{proposition: speicla case} and Definition \ref{d.bigO}.} In particular, $Q \approx_{\nil} Q_{\bullet \rightarrow \bullet}$, as in \eqref{eq: rooted tree and a loop}. 
\end{cor} 

\begin{proof} 
Set $Q' = Q_{\bullet \rightarrow \bullet}$. Since  $\NI_Q(d) = O(r^d)$, it immediately follows that $Q \le_{\textrm{nil}} Q'$. Conversely, suppose without loss of generality that $T_1$ has at least two vertices. Then $Q$ has a subquiver of the form  
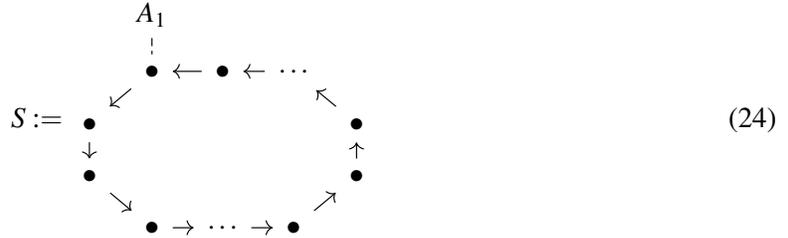
\begin{equation}
S:= \begin{tikzcd}[row sep =0.6em, column sep =0.5em]
	& A_1 \arrow[d,no head, dashed] &  &  &  \\
	 & \bullet \arrow[dl] & \bullet \arrow[l] & \cdots \arrow[l]&  \\
	\bullet \arrow[d] &  &  &  & \bullet \arrow[ul]\\
	\bullet \arrow[dr] &  &  &  & \bullet \arrow[u] \\
	 & \bullet \arrow[r] & \cdots  \arrow[r] & \bullet \arrow[ur] &   \\
\end{tikzcd}    
\end{equation} 
where $A_1 = \bullet \xrightarrow[]{} \bullet$ or $\bullet \xleftarrow[]{} \bullet$ (considered as rooted trees). If the central cycle of $S$ has $n$ vertices, then $\NI_{Q'}(d) \le \NI_S(nd)$ for all $d$. Indeed, let $M$ be a $d$-dimensional, nilpotent indecomposable $Q'$-representation $M$ with spine 
\[
s_1\xrightarrow[]{\alpha_1}\cdots s_{j-1}\xrightarrow[]{\alpha_{j-1}}s_j.
\]
In particular, $M$ has spine length $j$ and $j \le d \le 2j$. Define $\tilde{M}$ to be the $Q$-representation whose coefficient quiver is obtained from the following process: 
\begin{enumerate} 
\item Replace each arrow in the spine of $\Gamma_M$ with a oriented path containing $n$ vertices, so that the spine becomes 
\[
s_{1,1}\xrightarrow[]{\alpha_{1,1}}\cdots s_{1,n-1}\xrightarrow[]{\alpha_{1,n-1}}s_{1,n}\xrightarrow[]{\alpha_{2,1}}s_{2,1}\xrightarrow[]{\alpha_{2,1}}\cdots s_{j,n-1}\xrightarrow[]{\alpha_{j,n-1}}s_{j,n}.
\] The result is a $Q'$-representation with dimension $nj + d-j$. Since $j \le d$,  
\[
nj+d-j = j(n-1)+d \le d(n-1)+d = nd. 
\]
\item After possibly reversing the orientation of the arrows off the spine, this $Q'$-representation can naturally be considered an $S$-representation.
\item Add arrows to the end of the spine to obtain the coefficient quiver $\Gamma_{\tilde{M}}$, so that the spine of the $S$-representation becomes 
\[
s_{1,1}\xrightarrow[]{\alpha_{1,1}}\cdots s_{1,n-1}\xrightarrow[]{\alpha_{1,n-1}}s_{1,n}\xrightarrow[]{\alpha_{2,1}}s_{2,1}\xrightarrow[]{\alpha_{2,1}}\cdots s_{j,n-1}\xrightarrow[]{\alpha_{j,n-1}}s_{j,n} \xrightarrow[]{}s_{jn+1} \xrightarrow[]{}\cdots \xrightarrow[]{} s_{(n-1)(d-j)}.
\] 
The result is an indecomposable $S$-representation of dimension $dn$. 
\end{enumerate} 
It is easy to verify that the map $M \mapsto \tilde{M}$ is injective, and so $\NI_{Q'}(d) \le \NI_S(nd)$. In particular, we have $Q' \le_{\textrm{nil}} S \le_{\textrm{nil}} Q$. It follows that $Q' \approx_{\textrm{nil}}Q$, as we wished to show.
\end{proof}

\begin{myeg}
Consider the following quivers. 
\begin{equation}\label{eq: cycle and arrow}
	Q_1:=\begin{tikzcd}
		\bullet \arrow[d]\\
		\bullet 
		\arrow[out=240,in=300,loop,swap]
	\end{tikzcd} \qquad Q_2= \begin{tikzcd}
	\bullet \arrow[r]	& \bullet \arrow[r,out=300,in=250] & \bullet \arrow[l,out=120,in=60]  & \bullet \arrow[l]
\end{tikzcd} 
\end{equation}
We know that $NI_{Q_1}(d)=O((\frac{1+\sqrt{5}}{2})^d)$ from Example \ref{example: one arrow and two arrows}. On the other hand, for $Q_2$, we obtain the following recursive formula:
\[
g(d)= b_2g(d-2) + b_3g(d-3)+b_4g(d-4) =g(d-2)+2g(d-3)+g(d-4),
\]
where $g(d)=Q_{i \to f}(d)$ as in Definition \ref{definition: recursion count}. So, we have $p(x)=x^4-x^2-2x-1$, or $x^2=x+1$ and $c=\frac{1+\sqrt{5}}{2}$. Hence, we have $Q_1 \approx_{nil} Q_2$.
\end{myeg}

\begin{myeg}
\[
Q_2'= \begin{tikzcd}
	\bullet \arrow[r]	& \bullet \arrow[r,out=300,in=250] & \bullet \arrow[l,out=120,in=60] 
\end{tikzcd} 
\]	
For $Q_2'$, we obtain the following recursive formula:
\[
g(d)= g(d-2)+ g(d-3) 
\]
Hence, we have $x^3-x-1=0$ and $c \approx 1.32$, showing that $Q_2' \approx_{nil} Q_2 \approx_{\nil} Q_1$ from the previous example.
\end{myeg}

\subsection{Pseudotrees with acyclic central cycle}  \label{subsection: reversing arrows}
Let $Q$ be a proper pseudotree and $n$ a natural number. Let $\mathcal{N}_Q(n)$ be a set containing exactly one representative of each isomorphism class of $n$-dimensional indecomposable object in $\operatorname{Rep}(Q,\mathbb{F}_1)_{\nil}$. Then we can decompose $\mathcal{N}_Q(n)$ as follows: 
\[  
\mathcal{N}_Q(n) = \mathcal{T}_Q(n)\sqcup \mathcal{P}_Q(n),
\] 
where
\[
\mathcal{T}_Q(n) = \{ M \in \mathcal{N}_Q(n)\mid H_1(\Gamma_M,\mathbb{Z}) = 0 \} \textrm{ and } \mathcal{P}_Q(n) = \{ M \in \mathcal{N}_Q(n) \mid H_1(\Gamma_M,\mathbb{Z}) = \mathbb{Z}\}.
\]
These sets consist of the $n$-dimensional tree representations and proper pseudotree representations of $Q$. Note that $Q$ admits proper pseudotree representations if and only if its central cycle is acyclic (see \cite{jun2021coefficient}).
For an arrow $\alpha \in Q_1$, let $r_{\alpha}Q$ denote the quiver which is obtained from $Q$ by reversing the orientation of $\alpha$. Given a representation $M$ with winding map $c_M : \Gamma_M \rightarrow Q$, we obtain a winding $r_{\alpha}(c_M) : r_{\alpha}(\Gamma_M) \rightarrow r_{\alpha}Q$ by reversing the orientation of each $\alpha$-colored arrow in $\Gamma_M$. Denote the corresponding representation of $r_{\alpha}Q$ by $r_{\alpha}M$ (i.e. $\Gamma_{r_{\alpha}M} = r_{\alpha}(\Gamma_M)$ and $c_{r_{\alpha}M} = r_{\alpha}(c_M)$). 
\begin{rmk} 
By a slight abuse of notation, we denote the arrow corresponding to $\alpha$ in $r_{\alpha}Q$ by $\alpha$. 
\end{rmk} 
\noindent The following proposition is straightforward. 
\begin{pro}\label{p.reflection}
Let $Q$ be a quiver and $M$ be an $\mathbb{F}_1$-representation of $Q$. Then the following hold: 
\begin{enumerate} 
\item $r_{\alpha}r_{\alpha}M \cong M$.
\item $\underline{\operatorname{dim}}(r_{\alpha}M) = \underline{\operatorname{dim}}(M)$. 
\item $r_{\alpha}M$ is indecomposable if and only if $M$ is indecomposable.  
\item $H_1(\Gamma_{r_{\alpha}M},\mathbb{Z}) \cong H_1(\Gamma_M,\mathbb{Z})$.
\end{enumerate}
\end{pro}  
Now let $Q$ be a proper pseudotree. If $M \in \mathcal{T}_Q(n)$, then Proposition \ref{p.reflection}(2)-(4) implies that $r_{\alpha}M$ is isomorphic to exactly one element of $\mathcal{T}_{r_{\alpha}Q}(n)$. Therefore, we have a map $r_{\alpha,n} : \mathcal{T}_{Q}(n) \rightarrow \mathcal{T}_{r_{\alpha}Q}(n)$ for each $n$. By Proposition \ref{p.reflection}(1), $r_{\alpha,n}$ is bijective. We have the immediate corollary: 
\begin{cor} 
Let $Q$ be a proper pseudotree. Then for each natural number $n$, the number $|\mathcal{T}_Q(n)|$ does not depend on the orientation of $Q$. 
\end{cor} 
\begin{proof} 
Let $Q'$ be a proper pseudotree with the same underlying graph as $Q$. Then there exist arrows $\alpha_1,\ldots , \alpha_m \in Q_1$ such that $Q' = [r_{\alpha_m}r_{\alpha_{m-1}}\cdots r_{\alpha_1}]Q$. The claim then follows from the bijection above.  
\end{proof} 
\begin{cor} 
Let $Q$ be a proper pseudotree. Then there exists a positive real number $a>1$ such that $|\mathcal{T}_Q(n)| = O(a^n)$. 
\end{cor} 
\begin{proof} 
This is true if the central cycle of $Q$ is equioriented. By the previous corollary, it must also be true for $Q$. 
\end{proof}
Note that the corresponding result does not hold for $\mathcal{P}_Q(n)$: $r_{\alpha}M$ might not be nilpotent, even if $M$ is. Nevertheless, we still have the following result. 
\begin{pro} 
Let $Q$ be a proper pseudotree. Then for each $n$, $|\mathcal{P}_Q(n)| \le |\mathcal{T}_Q(n)|$. 
\end{pro} 
\begin{proof} 
Without loss of generality, we assume $\mathcal{P}_Q(n) \neq \emptyset$. Let $C$ denote the central cycle of $Q$. For each $N \in \mathcal{P}_Q(n)$, $c_N^{-1}(C) \neq \emptyset$, and so we may choose an arrow $\alpha_N \in c_N^{-1}(C)$. Let $T(N)$ denote the representation obtained from $N$ by deleting $\alpha_N$ from $\Gamma_N$ (but not its source or target). Then $T(N) \in \mathcal{T}_Q(n)$. This implies the existence of a surjective function $P: \mathcal{T}_Q(n) \rightarrow \mathcal{P}_Q(n)$, which we construct as follows: fix any representation $Z \in \mathcal{P}_Q(n)$. For any $M \in \mathcal{T}_Q(n)$, $c_M^{-1}(C)$ is an orientation of a path, or a single vertex, or empty. In the latter two cases, set $P(M) = Z$. Otherwise, $c_M^{-1}(C)$ contains two vertices $\ell_1, \ell_2$ which are adjacent to exactly one vertex in $c_M^{-1}(C)$\footnote{In the equioriented case, these are simply the start and finish of the spine.}. There is at most one arrow $\alpha \in Q$ between $c_M(\ell_1)$ and $c_M(\ell_2)$: if this arrow exists, let $P(M)$ be the representation obtained by adding an $\alpha$-colored arrow between $\ell_1$ and $\ell_2$; otherwise, let $P(M) = Z$. For each $N \in \mathcal{P}_Q(n)$, $T(N)$ satisfies $P(T(N)) = N$. Hence, $P$ is a surjection. 
\end{proof} 

\begin{cor} 
Let $Q$ be a proper pseudotree. 
In particular, $Q \approx_{\nil} Q'$ if and only if $Q'$ is a proper pseudotree. 
\end{cor} 

\begin{proof} 
For each $n$, we have that
\[ 
\operatorname{NI}_Q(n) = |\mathcal{N}_Q(n)| = |\mathcal{T}_Q(n)| + |\mathcal{P}_Q(n)| = O(a^n),
\] 
with $a >1$. Hence, $Q \le_{\textrm{nil}} Q_{\bullet \xrightarrow[]{} \bullet}$. Conversely, let $Q'$ be a quiver with the same underlying graph as $Q$ but with an equioriented central cycle. Then $\NI_{Q'}(n) = |\mathcal{T}_Q(n)| \le \NI_Q(n)$ for all $n$, and so $Q' \le_{\textrm{nil}} Q$. But $Q_{\bullet \xrightarrow[]{} \bullet} \le_{\textrm{nil}} Q'$ by Corollary \ref{corollary: main cor}, from which the claim follows.
\end{proof}

\subsection{The growth of nilpotent indecomposables} 

In \cite[Questions 5.17, 5.18]{jun2020quiver} the first and third authors posed the following questions: 
\begin{enumerate} 
\item Let $Q$ and $Q'$ be proper pseudotrees. Does it follow that $Q\approx_{\nil} Q'$?
\item Are there any pseudotrees $Q$, with $Q \approx_{\nil} \wild_2$? Here, $\wild_2$ is the quiver with one vertex and two loops.
\end{enumerate} 
The results above show that the first question is true. It turns out that the second is false, because $\textrm{NI}_{\wild_2}(n)$ grows faster than any exponential function. To be precise, we have the following.

\begin{pro}
Let $Q=\wild_2$ be the two loop quiver. Then, $\NI_Q(n)\geq O((n/2)!)$. 
\end{pro}
\begin{proof}	
We may assume that $n$ is an even number, say $n=2k$. We show that there exists a set of distinct isomorphism classes whose cardinality is $k!=(n/2)!$. We use the red color for one loop and the blue color for the other. To construct a coefficient quiver for $\wild_2$, there are total $2k$ vertices: 
\begin{equation}\label{eq: verticesonly}
		\begin{tikzcd}[arrow style=tikz,>=stealth,row sep=2em]
	2k    & (2k-1)   & \cdots &4  & 3  &    2   & 1
\end{tikzcd}
\end{equation}
We may first a string with $k$ red arrows as follows:
\begin{equation}\label{qqq}
	\begin{tikzcd}[arrow style=tikz,>=stealth,row sep=2em]
		2k   \arrow[r,red, bend right, swap,thick, dashed] & (2k-1) \arrow[r,red, bend right, swap,thick, dashed]  & \cdots \arrow[r,red, bend left, swap,thick, dashed] &4 \arrow[r,red, bend left, swap,thick, dashed]  & 3  \arrow[r,red,thick, dashed] &    2  \arrow[r,red, bend right, swap,thick, dashed]  & 1
	\end{tikzcd}
\end{equation}
Then, we add $k$ blue arrows as follows: we permute $A=\{2k,2k-1,2k-2,\dots,k+1\}$ (resp.~$B=\{k,k-1,\dots,1\}$) to make $A=\{a_1,\dots,a_{k+1}\}$ (resp.~$B=\{b_1,\dots,b_k\}$) and draw a blue arrow from $a_i$ to $b_i$. In this way, we obtain a class of coefficient quivers. One can check the number of such coefficient quivers is at least $k!=(n/2)!$, considering isomorphism classes. This proves our claim. 
\end{proof} 

This allows us to completely classify (finite connected) quivers with respect to the equivalence relation $\approx_{\textrm{nil}}$. More precisely, we have the following:

\begin{cor} 
Let $Q$ be a finite connected quiver. Then $Q \approx_{\textrm{nil}} Q'$, where $Q'$ is exactly one of the following four quivers: 
\begin{enumerate} 
\item $\wild_0 = \bullet$
\item $\wild_1 = \begin{tikzcd} \bullet \arrow[out=30,in=-30,loop] \end{tikzcd}$
\item $Q_{\bullet \xrightarrow[]{} \bullet} = \begin{tikzcd}
	\bullet \arrow[r] \arrow[out=150,in=210,loop] &	\bullet 
\end{tikzcd}$
\item $\wild_2 = \begin{tikzcd} \bullet \arrow[out=30,in=-30,loop] \arrow[out=150,in=210,loop] \end{tikzcd}$
\end{enumerate}  
Furthermore, $Q \approx_{\textrm{nil}} \wild_0$ if and only if $Q$ is a tree, $Q \approx_{\textrm{nil}} \wild_1$ if and only if $Q$ is of type $\tilde{\mathbb{A}}$, $Q \approx_{\textrm{nil}} Q_{\bullet \xrightarrow[]{} \bullet}$ if and only if $Q$ is a proper pseudotree, and $Q \approx_{\textrm{nil}} \wild_2$ otherwise.
\end{cor}

Using the terminology of \cite{jun2020quiver}, $Q\le \wild_0$ if and only if $Q$ has finite representation type over $\mathbb{F}_1$, and $Q \le \wild_1$ if and only if $Q$ has bounded representation type over $\mathbb{F}_1$. As a consequence of the results above, $Q \le_{\textrm{nil}} Q_{\bullet \xrightarrow[]{} \bullet}$ if and only if $\NI_Q = O(a^n)$ for some $a>1$ and $Q \approx_{\textrm{nil}} Q_{\bullet \xrightarrow[]{} \bullet}$ if and only if $\NI_Q \approx_{\textrm{nil}} a^n$ for some $a>1$.\footnote{See Remark \ref{r.functions} above.}


\section{Lie algebras of pseudotrees} \label{section: Lie algebra}

In this section, we let $Q$ be a (not necessarily proper) pseudotree unless otherwise stated. Let $\mathcal{T}_Q := \bigcup_{n \in \mathbb{N}}{\mathcal{T}_Q(n)}$ and  $\mathcal{P}_Q := \bigcup_{n \in \mathbb{N}}{\mathcal{P}_Q(n)}$, where $\mathcal{T}_Q(n)$ (resp.~$\mathcal{P}_Q(n)$) is the set of isomorphism classes of $n$-dimensional, nilpotent indecomposable tree (resp.~proper pseudotree) representations of $Q$. 

Let $\mathfrak{n}:=\mathfrak{n}_{Q,\nil}$ denote the Lie algebra of nilpotent indecomposable representations as in Section \ref{subsection: Hall algebras}, and $\mathfrak{p}:=\textrm{Span}(\mathcal{P}_Q)$. The following are straightforward from the definitions. 

\begin{lem}\label{l: ideal}
	With the same notation as above, the following hold. 
	\begin{enumerate}
		\item 
		If $0 \rightarrow X \rightarrow E \rightarrow Y \rightarrow 0$ is a short exact sequence in $\operatorname{Rep}(Q,\FF_1)_{\nil}$ and either $X$ or $Y$ is in $\mathcal{P}_Q$, then $E \in \mathcal{P}_Q$ as well. In particular, $\mathfrak{p}$ is a Lie ideal of $\mathfrak{n}$.
		\item 
		If $X$ and $Y$ are both in $\mathcal{P}_Q$, then $E \cong X\oplus Y$. In particular, $\mathfrak{p}$ is abelian.	
	\end{enumerate}
\end{lem}
\begin{proof}
	The first assertion directly follows from Lemma \ref{lemma: lemma from p1}. The second assertion is also clear since any indecomposable nilpotent representation of a pseudotree $Q$ is either in $\mathcal{T}_Q$ or $\mathcal{P}_Q$. 
\end{proof}

The quotient $\mathfrak{n}/\mathfrak{p}$ is a Lie algebra with a basis consisting of the cosets of elements in $\mathcal{T}_Q$.  By Proposition \ref{p.reflection}, there is a bijection $r : \mathcal{T}_Q \rightarrow \mathcal{T}_{Q'}$ for any quiver $Q'$ with underlying graph $\overline{Q}$, and hence the underlying vector space of $\mathfrak{n}/\mathfrak{p}$ only depends on the underlying graph $P:=\overline{Q}$. It will be convenient to write $\mathfrak{q}_P:= \mathfrak{n}/\mathfrak{p}$ to denote the underlying vector space of $\mathfrak{n}/\mathfrak{p}$. The Lie bracket on $\mathfrak{n}/\mathfrak{p}$ is then thought of as a linear map $[\cdot,\cdot]_Q : \mathfrak{q}_P \otimes \mathfrak{q}_P \rightarrow \mathfrak{q}_P$. 

\begin{mydef}
	Let $Q$ be a proper pseudotree and $C$ be the central cycle of $Q$. A representation $T \in \mathcal{T}_Q$ is called a \emph{spine} representation if $c_T^{-1}(C_0) \neq \emptyset$, where $c_T:\Gamma_T\to Q$ is a associated winding and $C_0$ is the set of vertices of $C$. Otherwise, $T$ is called a \emph{branch} representation. 
\end{mydef}


If $T,S \in \mathcal{T}_Q$, then the non-split short exact sequences 
\[ 
0 \rightarrow S \rightarrow E \rightarrow T \rightarrow 0
\] 
turn out to be quite constrained. Note that the following result can be sharpened by placing additional assumptions on the tree representations involved, but in its present form, it is strong enough for the purposes of this article. For this result, we write $Q$ as
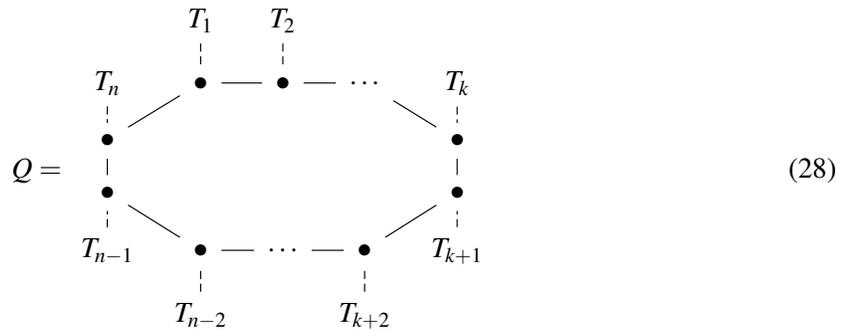
\begin{equation}\label{eq: branches}
Q = \begin{tikzcd}[row sep =0.6em, column sep =0.5em]
	& T_1 \arrow[d,no head, dashed] & T_2 \arrow[d,no head, dashed] &  &  \\
	T_n \arrow[d,no head, dashed] & \bullet \arrow[dl, no head] & \bullet \arrow[l, no head] & \cdots \arrow[l,no head]& T_k \arrow[d,no head, dashed] \\
	\bullet \arrow[d, no head] &  &  &  & \bullet \arrow[ul, no head]\\
	\bullet \arrow[dr, no head] &  &  &  & \bullet \arrow[u, no head] \\
	T_{n-1} \arrow[u,no head, dashed] & \bullet \arrow[r, no head] & \cdots  \arrow[r, no head] & \bullet \arrow[ur, no head] & T_{k+1} \arrow[u,no head, dashed]  \\
	& T_{n-2} \arrow[u,no head, dashed]  & & T_{k+2} \arrow[u,no head, dashed]   \\
\end{tikzcd}     
\end{equation} 
where the orientation of each arrow can be arbitrary, and where each $T_i$ is a rooted tree.

\begin{lem}\label{l: tree ext}
	Let $Q$ be a pseudotree with central cycle $C$, and $S, T \in \mathcal{T}_Q$. Then the following hold: 
	\begin{enumerate} 
		\item If either $S$ or $T$ is a branch representation, then there are at most $\max\{ 1, |c_S^{-1}(C_0)|, |c_T^{-1}(C_0)|\}$ non-split short exact sequences $0 \rightarrow S \rightarrow E \rightarrow T \rightarrow 0$. In the case that such a sequence exists, $E \in \mathcal{T}_Q$ and any short exact sequence $0 \rightarrow T \rightarrow E' \rightarrow S \rightarrow 0$ splits. 
		\item If $S$ and $T$ are both spine representations, then there are either zero, one, or three non-split short exact sequences $0 \rightarrow S \rightarrow E \rightarrow T \rightarrow 0$. If there is exactly one non-split sequence, then $E \in \mathcal{T}_Q$ and there is at most one non-split short exact sequence $0 \rightarrow T \rightarrow E' \rightarrow S \rightarrow 0$. If there are exactly three such sequences, $E \in \mathcal{T}_Q$ for two sequences, $E \in \mathcal{P}_Q$ for the third, and any short exact sequence $0 \rightarrow T \rightarrow E' \rightarrow S \rightarrow 0$ splits. 
		\item In $\mathfrak{n}/\mathfrak{p}$, $[S,T]$ can be written as a sum of at most $\max\{ 2, |c_S^{-1}(C_0)|, |c_T^{-1}(C_0)| \}$ elements of $\mathcal{T}_Q$ with coefficients in $\mathbb{Z}$. 
	\end{enumerate}
\end{lem} 

\begin{proof} 
	Let $0 \rightarrow S \rightarrow E \rightarrow T \rightarrow 0$ be a non-split short exact sequence. Then $\Gamma_E$ is obtained from the disjoint union $\Gamma_S \sqcup \Gamma_T$ by adding certain $\alpha$-colored arrows from $\alpha$-sinks of $\Gamma_T$ to $\alpha$-sources of $\Gamma_S$, where $\alpha \in Q_1$. Since $\Gamma_E$ is a not-necessarily proper pseudotree, at most $2$ arrows can be added from $\Gamma_T$ to $\Gamma_S$.  
	
	\noindent (1): If either $S$ or $T$ is a branch representation, then from Figure \eqref{eq: branches} one can observe that any arrow connecting $\Gamma_T$ to $\Gamma_S$ must necessarily lie in $c_E^{-1}(Q_1\setminus C_1)$. Hence, the connecting arrow must lie in $T_i$ for a unique $i$, and connect one vertex of $T_i$ to another. In particular, if there exists a non-split short exact sequence $0 \rightarrow S \rightarrow E \rightarrow T \rightarrow 0$, then $E$ is obtained by adding exactly one arrow from $\Gamma_T$ to $\Gamma_S$. In this case, there will be no arrow from $\Gamma_S$ to $\Gamma_T$, and so any short exact sequence $0 \rightarrow T \rightarrow E' \rightarrow S \rightarrow 0$ will split. Now, since only one arrow is added to $\Gamma_S \sqcup \Gamma_T$, $\Gamma_E \in \mathcal{T}_Q$, and the number of non-split short exact sequences is at most $\max\{ 1, c_S^{-1}(C_0), c_T^{-1}(C_0)\}$. This proves the first assertion. 
	
	\noindent(2): Suppose that $S$ and $T$ are both spine representations. Since at most two arrows may be added, there are at most $3$ ways to add these arrows which result in a connected quiver $\Gamma_E$. Furthermore, the only time when there are exactly three non-split short exact sequences $0 \rightarrow S \rightarrow E \rightarrow T \rightarrow 0$ is when two arrows may be added from $\Gamma_T$ to $\Gamma_S$. Adding both arrows results in a quiver which is a proper pseudotree, whereas adding only one arrow results in an element in $\mathcal{T}_Q$. In this case there is no way to add an arrow from $\Gamma_S$ to $\Gamma_T$, and so any short exact sequence $0 \rightarrow T \rightarrow E' \rightarrow S \rightarrow 0$ will split. Otherwise only one arrow may be added, in which case the resulting representation will be in $\mathcal{T}_Q$. This proves the second assertion. 
	
	\noindent (3): The last assertion follows from $(1)$ and $(2)$.
\end{proof}  

In general, the task of explicitly computing products $[N]\cdot[M] \in H_{Q,\nil}$ or commutators $[[N],[M]] \in\mathfrak{n}_{Q,\nil}$ is expected to be challenging. It is not even clear how many non-zero terms such expressions will have, when expanded as a linear combination of $\mathbb{F}_1$-representations of $Q$. The results in this section give us some insight into the behavior of commutators, at least in the case that $Q$ is a not-necessarily-proper pseudotree. 

\begin{cor} 
Let $Q$ be a not-necessarily proper pseudotree, and let $M$ and $N$ be indecomposable nilpotent $\mathbb{F}_1$-representations of $Q$. Then the following hold. 
\begin{enumerate} 
\item If $Q$ is a tree, then there exists a natural number $K$ such that $\max\{\dim(M),\dim(N) \} \geq K$ implies $[[N],[M]] = 0$.
\item If $Q$ is of type $\tilde{\mathbb{A}}$, then there exist natural numbers $K, B$ such that $\max\{\dim(M),\dim(N) \} \geq K$ implies $[[N],[M]]$ is a $\mathbb{Z}$-linear combination of at most $B$ indecomposable nilpotent representations. 
\item If $Q$ is a proper pseudotree, then there exists a natural number $K$ such that \\
 $\max\{\dim(M),\dim(N) \} \geq K$ implies $[[N],[M]]$ is a $\mathbb{Z}$-linear combination of at most \\
 $\max\{\dim(M),\dim(N) \}$ indecomposable nilpotent representations. 
\end{enumerate}
\end{cor}

\begin{proof} 
Assertion (1) follows from the fact that if $Q$ is a tree, then $Q$ is of finite representation type over $\mathbb{F}_1$. Similarly, assertion (2) follows from the fact that if $Q$ is of type $\tilde{\mathbb{A}}$, then $Q$ is of bounded representation type over $\mathbb{F}_1$. See \cite{jun2020quiver} for more details. For assertion (3), first note that Lemma \ref{l: ideal}(2) implies that we may assume without loss of generality $M \in \mathcal{T}_Q$. If $N \in \mathcal{P}_Q$, then $c_N^{-1}(\alpha ) \neq \emptyset$ for any $\alpha \in C_1$, where $C$ is the central cycle of $Q$. Fixing an arrow $\alpha \in C_1$ and deleting exactly one $\alpha$-colored arrow from $\Gamma_N$, we obtain a representation $N' \in \mathcal{T}_Q$ such that $[[N'],[M]]$ and $[[N],[M]]$ have the same number of non-zero terms. Hence, we may also assume without loss of generality that $N \in \mathcal{T}_Q$. Then by Lemma \ref{l: tree ext}(3) $[[N],[M]]$ is a sum of at most $\max\{3,|c_M^{-1}(C_0)|, |c_N^{-1}(C_0)|\}$ nilpotent indecomposable representations, and the result follows.
\end{proof}

In light of the above result, it is natural to pose the following question.
\begin{question}   
	Let $Q$ be a connected quiver, with $M$ and $N$ indecomposable nilpotent $\mathbb{F}_1$-representations. What is an upper bound for $|\operatorname{supp}([M,N])|$? 
\end{question} 

The question above is posed with the understanding that nilpotent indecomposable representations form a basis for $\mathfrak{n}$. In particular, they generate $\mathfrak{n}$: as the following result demonstrates, however, they are not necessarily a minimal generating set. The theorem below proves some fundamental structural results about the Lie algebras $\mathfrak{n}$ and $\mathfrak{n}/\mathfrak{p}$. This result directly extends \cite[Propositions 6.3, 6.7]{jun2021coefficient} to the case of proper pseudotrees.

\begin{mythm} \label{theorem: Lie main thm}
	Let $Q$ be a not-necessarily proper pseudotree. Let $\mathcal{T}_Q$, $\mathcal{P}_Q$, $\mathfrak{n}$ and $\mathfrak{n}/\mathfrak{p}$ be as defined above. Then the following hold: 
	\begin{enumerate} 
		\item As a Lie algebra, $\mathfrak{n}$ is generated by $\mathcal{T}_Q$. 
		\item As a Lie algebra, the isomorphism class of $\mathfrak{n}/\mathfrak{p}$ does not depend on the orientation of $Q$.
	\end{enumerate}
\end{mythm} 
\begin{proof} 
We first prove assertion (1). If $Q$ is a tree, then there is nothing to show. Suppose that $Q$ is a proper pseudotree and $C$ is the central cycle of $Q$. If $C$ is equioriented then $\mathcal{P}_Q = \emptyset$ and there is nothing to show.
	
Suppose that $C$ is acyclic. Then $C$ must contain a sink $v \in C_0$, which is the target of two distinct arrows $\alpha, \beta \in C_1$. Let $M \in \mathcal{P}_Q$. Then $c_M^{-1}(x) \neq \emptyset$ for all $x \in C_0$. In particular, there exists a vertex $u \in c_M^{-1}(v)$. Let $\tilde{\alpha}$ and $\tilde{\beta}$ be the arrows in $c_M^{-1}(C_1)$ which are adjacent to $u$. Then the quiver $\Gamma$ obtained from $\Gamma_M$ by deleting $\tilde{\alpha}$, $\tilde{\beta}$ and $u$ is the coefficient quiver of a tree representation $T \in \mathcal{T}_Q$. Let $T' \in \mathcal{T}_Q$ be the representation whose coefficient quiver is the full subquiver of $\Gamma_M$ with vertex set $(\Gamma_M)_0 \setminus T_0$. Then  
	\[ 
	[T]\cdot [T'] = [S_v \oplus T] + [T\cup\tilde{\alpha}] + [T\cup\tilde{\beta}] + [M],
	\] 
	where $T\cup \tilde{\alpha}$ (resp.~$T\cup \tilde{\beta}$) is the coefficient quiver obtained by gluing back $\tilde{\alpha}$ (resp.~$\tilde{\beta}$) to $T$. We also have the following. 
	\[  
	[T']\cdot [T] = [T' \oplus T].
	\] 
	It follows that in $\mathfrak{n}$, 
	\[ 
	[[T],[T']] = [T\cup\tilde{\alpha}] + [T\cup\tilde{\beta}] + [M].
	\] 
	Hence, we have that $[M] = [[T],[T']] - [T\cup\tilde{\alpha}] - [T\cup\tilde{\beta}]$, from which the first claim follows. 
	
	For the second claim, pick an arrow $\alpha \in Q_1$. Let $Q' = r_{\alpha}Q$, the quiver obtained from $Q$ by reversing the orientation of $\alpha$. It suffices to show that there exists a vector space isomorphism $\phi : \mathfrak{q}_P \rightarrow \mathfrak{q}_P$ satisfying 
	\begin{equation}\label{eq: lie alg}
	[\phi (S), \phi (T)]_{Q'} = \phi\left([ S,T]_Q \right)
	\end{equation}
	for all $S,T \in \mathcal{T}_Q$. For each $S \in \mathcal{T}_Q$, we set $\epsilon(S) = (-1)^{|c_S^{-1}(\alpha)|}$. This extends uniquely to define a linear map $\epsilon : \mathfrak{q}_P \rightarrow \mathbb{C}$. We claim that the linear map $\phi : \mathfrak{q}_P \rightarrow \mathfrak{q}_P$ defined by 
	\[ 
	\phi(S) = \epsilon(S)S,\quad \forall~S\in\mathcal{T}_Q
	\]  
	 is a vector space isomorphism satisfying \eqref{eq: lie alg}. To see this, pick an arrow $\beta \in Q_1$ and define 
	\begin{equation}\label{eq: vector iso}
	(\cdot )\xrightarrow[\beta]{Q} (\cdot ) : \mathfrak{q}_P\otimes \mathfrak{q}_P \rightarrow \mathfrak{q}_P
	\end{equation}
	as follows: for $S, T \in \mathcal{T}_Q$, we define $S\xrightarrow[\beta]{Q}T$ to be the sum of all coefficient quivers obtained by adding a $\beta$-colored arrow from $\Gamma_S$ to $\Gamma_T$. In particular, $S\xrightarrow[\beta]{Q}T = 0$ if no such coefficient quivers exist. 
	Then we may write 
	\begin{equation}\label{eq: lie bracket}
	[S,T]_Q = \sum_{\beta \in Q_1}{\left(S\xrightarrow[\beta]{Q}T - T\xrightarrow[\beta]{Q}S \right)}.
	\end{equation}  
	We now compute 
	\begin{align*}
		\phi\left([S,T]_Q \right) & = \phi \left(  \sum_{\beta \in Q_1}{\left(S\xrightarrow[\beta]{Q}T - T\xrightarrow[\beta]{Q}S \right)} \right) \\ 
		& = \sum_{\beta \in Q_1}{\left[ \epsilon\left(S\xrightarrow[\beta]{Q}T \right)S\xrightarrow[\beta]{Q}T - \epsilon\left(T\xrightarrow[\beta]{Q}S  \right) T\xrightarrow[\beta]{Q}S \right]}
	\end{align*} 
	and 
	\begin{align*} 
		[\phi(S),\phi(T)]_{Q'} & = \epsilon(S)\epsilon(T)[S,T]_{Q'} \\
		&= \sum_{\beta \in Q'_1}{\left(\epsilon(S)\epsilon(T)S\xrightarrow[\beta]{Q'}T - \epsilon(S)\epsilon(T)T\xrightarrow[\beta]{Q'}S \right)},
	\end{align*}  
	where we understand $\phi(S)$ and $\phi(T)$ as representations of $Q'$ by reversing the orientation of each $\alpha$-colored arrows in $\Gamma_S$ and $\Gamma_T$. We will be done if we can show the following
	\begin{equation}\label{eq:123}
		\epsilon\left(S\xrightarrow[\beta]{Q}T \right)S\xrightarrow[\beta]{Q}T - \epsilon\left(T\xrightarrow[\beta]{Q}S  \right) T\xrightarrow[\beta]{Q}S = \epsilon(S)\epsilon(T)S\xrightarrow[\beta]{Q'}T - \epsilon(S)\epsilon(T)T\xrightarrow[\beta]{Q'}S
	\end{equation}
	for all $\beta \in Q_1$. There are two cases.
	
	\noindent \underline{Case 1:} Suppose that $\beta \neq \alpha$ and $S\xrightarrow[\beta]{Q}T=\sum M_i$. Then for each $M_i$, we have
	\[
	\varepsilon(M_i)=(-1)^{|c_{M_i}^{-1}(\alpha)|}=(-1)^{|c_{S}^{-1}(\alpha)|+|c_{T}^{-1}(\alpha)|}=\varepsilon(S)\varepsilon(T).
	\] 
	It follows that 
	\[
	\epsilon\left(S\xrightarrow[\beta]{Q}T \right)S\xrightarrow[\beta]{Q}T = \sum \varepsilon (M_i)M_i =\sum \varepsilon(S)\varepsilon(T)M_i = \varepsilon(S)\varepsilon(T)\sum M_i= \epsilon(S)\epsilon(T)S\xrightarrow[\beta]{Q'}T. 
	\]
	Similarly, we have
	\[
	\epsilon\left(T\xrightarrow[\beta]{Q}S \right)T\xrightarrow[\beta]{Q}S = \epsilon(S)\epsilon(T)T\xrightarrow[\beta]{Q'}S.
	\]
	This proves \eqref{eq:123} in this case. 
	
	\noindent \underline{Case 2:} Suppose that $\beta = \alpha$,  $S\xrightarrow[\alpha]{Q}T=\sum M_i$, and $T\xrightarrow[\alpha]{Q}S=\sum N_i$. In this case, since we glued $S$ and $T$ (or $T$ and $S$) through the arrow $\alpha$, under the bijection in Section \ref{subsection: reversing arrows}, we have
	\[
	S\xrightarrow[\alpha]{Q'}T=\sum N_i, \quad T\xrightarrow[\alpha]{Q'}S=\sum M_i. 
	\] 
	Moreover, for each $M_i$, we have
	\[
	\varepsilon(M_i)=(-1)^{|c_{S}^{-1}(\alpha)|+|c_{T}^{-1}(\alpha)|+1}=-\varepsilon(S)\varepsilon(T).
	\] 
	Likewise, 
	\[
	\varepsilon(N_i)=-\varepsilon(S)\varepsilon(T).
	\]
	It follows that
	\[
	\epsilon\left(S\xrightarrow[\beta]{Q}T \right)S\xrightarrow[\beta]{Q}T = \sum \varepsilon (M_i)M_i =\sum -\varepsilon(S)\varepsilon(T)M_i = -\epsilon(S)\epsilon(T)T\xrightarrow[\beta]{Q'}S. 
	\]
	Similarly, 
	\[
	\epsilon\left(T\xrightarrow[\beta]{Q}S \right)T\xrightarrow[\beta]{Q}S = -\epsilon(S)\epsilon(T)S\xrightarrow[\beta]{Q'}T. 
	\]
	This proves \eqref{eq:123} in this case.  
\end{proof}

\begin{rmk}
Theorem \ref{theorem: Lie main thm}(1) shows that $\mathfrak{n}$ and $H_{Q,\nil}$ are generated by the tree representations. Within the literature, such representations may be interpreted as unramified tree modules over path algebras, see for instance \cite{lorscheid2015schubert}. Various notions of tree modules have been studied in the literature, see for instance \cite{krause1991maps, ringel1998exceptional, weist2010tree, weist2012tree, kinser2013trees, lorscheid2015schubert}. Such representations are of interest because they are \emph{absolutely indecomposable}, i.e. $M \in \mathcal{T}_Q$ implies $k\otimes_{\mathbb{F}_1}M$ is indecomposable for any field $k$. The properties of absolutely indecomposable $\mathbb{F}_1$-representations will be the topic of future investigations.  
\end{rmk}

In light of the above remark, it is natural to pose the following question.

\begin{question}  
	For which connected quivers $Q$ is $\mathfrak{n}$ generated (as a Lie algebra) by the tree representations?
\end{question}




\section{Coverings and Contractions} \label{section: coverings and contractions}


\subsection{Covering of a quiver}

For a given quiver $Q$, inspired by covering theory of quivers, we construct a family of coefficient quivers $c:\Gamma \to Q$ which have a nice grading distinguishing vertices. We first recall the definition of a covering of a quiver. 

\begin{mydef}\cite[Section 3]{franzen2021rationality}
Let $c:\Gamma \to Q$ be a quiver map. $c$ is said to be a covering if the following hold:
\begin{enumerate}
	\item 
$c$ is surjective on vertices and arrows. 
\item 
For any $i \in Q_0$ and $k \in \Gamma_0$ with $c_0(k)=i$, we have the following bijections:
\begin{equation}\label{eq: bij1}
\{\beta \in \Gamma_1 \mid s(\beta)=k\} \xleftrightarrow{1-1} \{\alpha \in Q_1 \mid s(\alpha)=i\},
\end{equation}
\begin{equation}\label{eq: bij2}
	\{\beta \in \Gamma_1 \mid t(\beta)=k\} \xleftrightarrow{1-1} \{\alpha \in Q_1 \mid t(\alpha)=i\}
\end{equation}
\end{enumerate}
A universal cover of $Q$ is a cover $c:\Gamma \to Q$, where $\Gamma$ is a tree. 
\end{mydef}

One can easily see that a covering defines an $\FF_1$-representation.

\begin{lem}\label{lemma: covering is winding}
Any covering $c:\Gamma \to Q$ is a winding. 
\end{lem}
\begin{proof}
Suppose that $c_1(\beta)=c_1(\beta')$ and $s(\beta)=s(\beta')=k$. Then, with $i=c_0(k)$, from \eqref{eq: bij1} we have $\beta=\beta'$. Likewise, when $t(\beta)=t(\beta')$, we have $\beta=\beta'$. 
\end{proof}

The following example shows that windings are strictly more general than coverings.

\begin{myeg}
\[
F:\Gamma=\left(
\begin{tikzcd}
	1 \arrow[r,"\alpha"] & 2 \arrow[r, "\beta"]& 3
\end{tikzcd}
\right)
\to Q=\left(
\begin{tikzcd}
	1 \arrow[r,"\alpha"] & 2 \arrow[r,shift right,swap, "\beta"] \arrow[r,shift left, "\gamma"]& 3
\end{tikzcd}
\right)
\]
In this case, we have $F_0(3)=3$, but $\{\beta \in \Gamma_1 \mid t(\beta)=3\} =\{\beta\}$ whereas $\{\alpha \in Q_1 \mid t(\alpha)=3)\}= \{\beta,\gamma\}$, in particular $F$ is not a covering. But, clearly $F$ is an winding.
\end{myeg}

The following is proven in \cite[Lemma 3]{franzen2021rationality}. 

\begin{lem}
Let $Q$ be a quiver. 
\begin{enumerate}
	\item 
Let $c:\Gamma \to Q$ be a universal cover and $d:D \to Q$ be any cover. Then, there exists a quiver map $p:\Gamma \to D$ such that $c=d\circ p$. 
\item 
If $Q$ is connected then $Q$ admits a universal cover. 
\end{enumerate}
\end{lem}



The following proposition characterizes quivers whose coefficients quivers have a non-degenerate nice grading. 

\begin{pro}
Let $Q$ be a quiver. $Q$ has no loop at a vertex if and only if any winding $c:\Gamma \to Q$ has a non-degenerate nice grading. 
\end{pro}
\begin{proof}
$(\Leftarrow):$ We first note $Q$ has no loop at a vertex if and only if $c:Q' \hookrightarrow Q$ has a non-degenerate nice grading for all subquivers $Q'$ of $Q$. Now, the first statement is clear as any inclusion $c:Q' \hookrightarrow Q$ is a winding. 

$(\Rightarrow):$ For a given winding $c:\Gamma \to Q$, we can directly define a nice grading $\partial$ on $\Gamma$ as follows. We first label vertices of $Q$ as $v_1,\dots, v_n$. For each vertex $w \in \Gamma_0$, we let $\partial(w)=i$ of and only if $c(w)=v_i$. Since $Q$ does not have a loop at any vertex, $\partial:\Gamma_0 \to \mathbb{Z}$ is a non-degenerate nice grading. 
\end{proof}

For any quiver $Q$, one can construct a winding $c:\Gamma \to Q$ which has a nice grading distinguishing vertices. This construction depends on a choice of an arrow of $Q$, and hence creates a family of windings distinguishing vertices. Note that this construction is a covering of $Q$, and any subquiver of $\Gamma$ has a nice grading distinguishing vertices. 

\begin{construction}\label{construction}
Let $Q$ be a quiver. Let's fix an arrow $e \in Q_1$ and we label vertices of $Q$ as $1,\dots, |Q_0|$. Let $\tilde{Q}_e$ be a quiver obtained from $Q$ by removing $e$ while keeping all vertices. Let $\ell \in \{1,\dots,|Q_0|-1\}$. Let $\Gamma_e$ be a quiver defined as follows:
\begin{equation}\label{eq:example constuction}
(\Gamma_e)_0:=\bigsqcup_{1\leq k \leq \ell} (\tilde{Q}_{e,k})_0, 
\end{equation}
where each $\tilde{Q}_{e,k}=\tilde{Q}_e$. For the arrows $(\Gamma_e)_1$, we keep all arrows of $\tilde{Q}_{e,k}$, and for each pair $(k,k+1)$, we add an arrow from a vertex $s(e)$ of $\tilde{Q}_{e,k}$ to $t(e)$ of $\tilde{Q}_{e,k+1}$, where we use the same labeling as $Q$.
\end{construction}

\begin{myeg}\label{example: family of ngdv}
Consider the following quiver (with the given labeling of vertices):
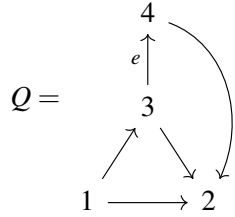
\[
Q=\begin{tikzcd}[column sep =0.7em]
 	&		4  \arrow[ddr,out=-20,in=65] & \\
	&	3 \arrow[dr] \arrow[u,"e"] & \\
	1 \arrow[ur] \arrow[rr]	& &2
\end{tikzcd}
\]
Then, $\Gamma_e$ as in \eqref{eq:example constuction} becomes the following:
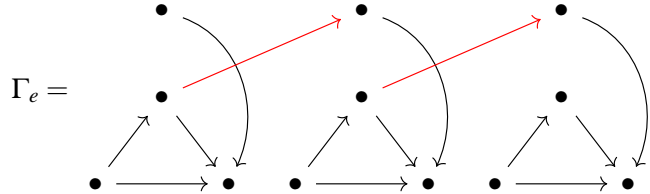
\[
\Gamma_e=\begin{tikzcd}[column sep =0.9em]
	&		\bullet  \arrow[ddr,out=-20,in=65] & & & \bullet  \arrow[ddr,out=-20,in=65]&& & \bullet  \arrow[ddr,out=-20,in=65]&\\
	&	\bullet \arrow[dr] \arrow[rrru, red]& &  & \bullet \arrow[rrru, red] \arrow[dr]& & &\bullet \arrow[dr]& \\
	\bullet \arrow[ur] \arrow[rr]	& &\bullet &  \bullet \arrow[ur] \arrow[rr]&  & \bullet  & \bullet \arrow[ur] \arrow[rr]& & \bullet&
\end{tikzcd}
\]
where the red arrows are newly added for $\Gamma_e$. Now, there is a natural winding $c:\Gamma_e \to Q$. Namely, we send $\tilde{Q}_{e,k}$ to $Q$ ``identically'', and send the red arrows to $e$. 
\end{myeg}

One can easily see, as in Example \ref{example: family of ngdv}, that there is a winding $c:\Gamma_e \to Q$ for any quiver $Q$, where $\Gamma_e$ is as in \eqref{eq:example constuction}. In fact, the following is straightforward from the definition. 

\begin{pro}
Let $Q$ be a quiver, and $c:\Gamma_e \to q$ be a natural winding as explained above. Then $\Gamma_e$ is a covering of $Q$. 
\end{pro}

Our next step is to equip $\Gamma_e$ is a nice grading distinguishing vertices as follows:
\begin{enumerate}
	\item 
Let $L$ be labeling of the vertices of $Q$. We assign an integer $m(\alpha)$ to each arrow $\alpha \in Q_1$ as $L(t(\alpha)) - L(t(\alpha))$ except $e$. We assign a ``large enough number'' to $e$ so that we can assign different integers to the vertices of the copies of $\tilde{Q}_e$ (see Example \ref{example: covering example}).
\item 
Now, we define $\partial:(\Gamma_e)_0 \to \mathbb{Z}$. For $\tilde{Q}_{e,1}$, we use $L$, i.e., for each $u \in (\tilde{Q}_{e,1})_0$, we define $\partial(u)=L(u)$. For $\tilde{Q}_{e,2}$, we let $\partial(t(e))=\partial(s(e))+m(e)$. Then, we define $\partial$ for the remaining vertices of $\tilde{Q}_{e,2}$ by using $m(\alpha)$ so that $m(\alpha)=\partial(t(\alpha)) - \partial(s(\alpha))$. 
\item 
Then, one iterate this process to define $\partial:(\Gamma_e)_0 \to \mathbb{Z}$. 
\end{enumerate}

\begin{myeg}\label{example: covering example}
The following is the case when $e=9$ from Example \ref{example: family of ngdv}. 
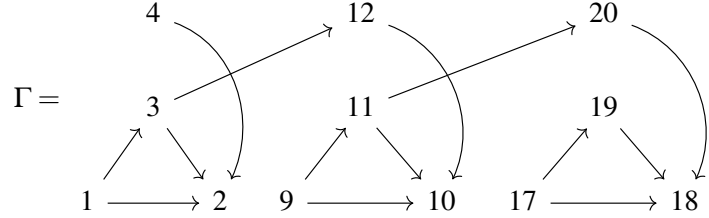
\[
\Gamma=\begin{tikzcd}[column sep =0.9em]
	&		4  \arrow[ddr,out=-20,in=65] & & & 12 \arrow[ddr,out=-20,in=65]&& & 20 \arrow[ddr,out=-20,in=65]&\\
	&	3 \arrow[dr] \arrow[rrru]& &  & 11 \arrow[rrru] \arrow[dr]& & &19 \arrow[dr]& \\
	1 \arrow[ur] \arrow[rr]	& &2 &  9 \arrow[ur] \arrow[rr]&  & 10 & 17 \arrow[ur] \arrow[rr]& & 18&
\end{tikzcd}
\]
\end{myeg}

From the above, we have the following proposition. 

\begin{pro}\label{proposition: main proposition}
Let $Q$ be a quiver, $e \in Q_1$, and $\Gamma_e$ be as in \eqref{eq:example constuction}. Then, the natural winding map $c:\Gamma_e \to Q$ has a nice grading which distinguishes vertices. 
\end{pro}

\begin{rmk}
	In fact, one could iterate Construction \ref{construction} as much as one wants. Also, one could remove more edges (rather than just one edge) to do the same construction under which Proposition \ref{proposition: main proposition} still holds. 	
\end{rmk}

\subsection{Contraction of coefficient quivers}

Next, we investigate how contractions and restrictions play a role in constructing a nice sequence distinguishing vertices. Let $Q$ be a quiver, and $c:\Gamma \to Q$ a winding. Let $A \subseteq Q_1$. Then, we have an induced winding $c|_A:\Gamma[c^{-1}(A)] \to Q$, where $\Gamma[c^{-1}(A)]$ is the induced subgraph on $c^{-1}(A)$. Also we have the following map:
\begin{equation}\label{eq: quotient map}
c/A:\Gamma/c^{-1}(A) \to Q/A. 
\end{equation}
The following example shows that \eqref{eq: quotient map} does not have to be a winding in general. 

\begin{myeg}
	\[
c:\Gamma=\left(
\begin{tikzcd}
	\bullet \arrow[r,"\alpha"] \arrow[d,swap,"\beta"] & \bullet \arrow[d,"\beta"] \\ 
	\bullet \arrow[r,swap,"\alpha"] & \bullet
\end{tikzcd}
\right)
\longrightarrow Q=\left(
\begin{tikzcd}
\bullet \arrow[out =330,in=390,loop,swap,"\beta"] \arrow[out=210,in=150,loop,"\alpha"]
\end{tikzcd}
\right)
\]
Let $A=\{\beta\}$. Then, we have the following:
\[
c/A:\Gamma/c^{-1}(A)= \left(
\begin{tikzcd}
 \bullet \arrow[r,shift right,swap, "\alpha"] \arrow[r,shift left, "\alpha"]& \bullet
\end{tikzcd}
\right) \longrightarrow Q/A=\left(\begin{tikzcd}
 \bullet \arrow[out=210,in=150,loop, "\alpha"] 
\end{tikzcd}\right)
\]
which is not an winding. 
\end{myeg}

Let $Q$ be a quiver, and $c:\Gamma \to Q$ a winding. For a nice grading $\partial:\Gamma_0 \to \mathbb{Z}$, one can obtain a function $\hat{\partial}:\Gamma_1 \to \mathbb{Z}$ defined by $\hat{\partial}(\alpha)=\partial (t(\alpha)) - \partial(s(\alpha))$ for $\alpha \in \Gamma_1$. The function $\hat{\partial}$ satisfies the condition \eqref{equation: winding}. Conversely, for a function $\hat{\partial}:\Gamma_1 \to \mathbb{Z}$ satisfying the condition \eqref{equation: winding}, one can find a corresponding nice grading $\partial$ up to addition of a constant. More generally, for a given nice sequence $\partial=\{\partial_i\}_{i \in \mathbb{N}}$, we obtain a sequence of functions $\{\hat{\partial}_i\}_{i \in \mathbb{N}}$, where $\hat{\partial}_i:\Gamma_1 \to \mathbb{Z}$ for each $i \in \mathbb{N}$ and they satisfies the condition $(3)$ in Definition \ref{definition: winding}. Conversely, a sequence of functions $\{\hat{\partial}_i\}_{i \in \mathbb{N}}$ recovers $\partial=\{\partial_i\}_{i \in \mathbb{N}}$ up to addition of constants.

\begin{mydef}
With the same notation as above, we say that $\partial=\{\partial_0,\dots,\partial_n\}$ frees $\alpha \in \Gamma_1$ if we can define a $\{\partial_0,\dots,\partial_n\}$-nice grading $\partial_{n+1}$ such that there is no $\beta \in \Gamma_1$ such that $c(\alpha)=c(\beta)$, $\partial_i(s(\alpha))=\partial_i(s(\beta))$, and $\partial_i(t(\alpha))=\partial_i(t(\beta))$ for $i=0,\dots,n+1$.
\end{mydef}

\begin{myeg}
	Let $Q = \wild_2$, with arrow set $Q_1 = \{\alpha_1,\alpha_2\}$. Let $M$ be the representation with quiver 
\[ 
\Gamma_M = 
\begin{tikzcd} 
	\bullet \arrow[r,"\alpha_2"] & \bullet \arrow[r,"\alpha_1"] & \bullet & \bullet \arrow[l,"\alpha_2",swap] & \bullet \arrow[l,"\alpha_1",swap] 
\end{tikzcd}
\]
In this case, any nice grading frees $\alpha_2$. 
\end{myeg}

\begin{rmk}\label{remark: free remark}
Let $c:\Gamma \to Q$ be a winding, and $\partial=\{\partial_0,\dots,\partial_n\}$ be a nice sequence distinguishing vertices. Then, $\partial$ frees all arrows in $\Gamma_1$ since the condition \eqref{eq: condition} cannot happen in this case. Conversely, if a nice sequence $\partial=\{\partial_0,\dots,\partial_n\}$ frees all arrows in $\Gamma_1$, then $\partial$ distinguishes vertices. 
\end{rmk}

\begin{pro}\label{proposition: contraction, restriction}
Let $Q$ be a quiver, and $c:\Gamma \to Q$ a winding. Let $A \subseteq Q_1$ such that
\[
c/A:\Gamma/c^{-1}(A) \to Q/A \quad \textrm{and} \quad c|_A:\Gamma[c^{-1}(A)] \to Q
\]
are windings which have a nice sequence distinguishing vertices. Then, $\Gamma$ has a nice sequence distinguishing vertices. 
\end{pro}
\begin{proof}
Let $\partial=\{\partial_0,\dots,\partial_k\}$ be a nice sequence distinguishing vertices for $c/A:\Gamma/c^{-1}(A) \to Q/A$. Each $\partial_i$ can be lifted to define a grading $\hat{\partial}_i:\Gamma_1 \to \mathbb{Z}$ by setting $\hat{\partial}_i(\alpha)=0$ for any $\alpha \in c^{-1}(A)$. Let $\partial_i':\Gamma_0 \to \mathbb{Z}$ be a grading obtained by $\hat{\partial}_i$. In this case, one can easily check that the lifted sequence of gradings $\partial' =\{\partial_0',\dots,\partial_k'\}$ is a nice sequence which frees the arrows of $\Gamma$ which are not in $c^{-1}(A)$. 

Now, we restrict $\Gamma$ to $\Gamma[c^{-1}(A)]$. From the assumption, we obtain a nice sequence distinguishing vertices $\{\partial_{k+1},\dots,\partial_n\}$. Again, this sequence can be lifted to define a nice sequence $\partial'' =\{\partial_{k+1}',\dots,\partial_n'\}$ which frees the arrows in $c^{-1}(A)$. In particular, by adding one more grading if needed, we may assume that $\partial_n(u) \neq \partial_n(v)$ for any vertices $u$ of $\Gamma/c^{-1}(A)$ and $v \in \Gamma[c^{-1}(A)]$. Now, one can easily check that $(\partial_0',\dots,\partial_n')$ is a nice sequence distinguishing vertices. 
\end{proof}

If $c:\Gamma \to Q$, $c/A:\Gamma/c^{-1}(A) \to Q/A$, and $c|_A:\Gamma[c^{-1}(A)]	\to Q$ are windings and $\Gamma$ has a nice sequence distinguishing vertices, then does $c/A$ also have a nice sequence distinguishing vertices? The following example shows that it is not the case. 

\begin{myeg}
Consider the following quivers $Q$ and $\Gamma$. 
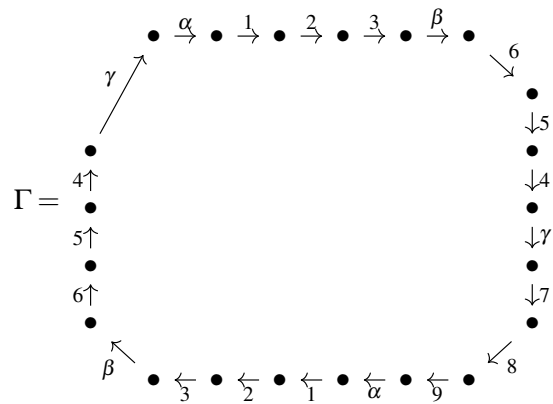
\[
	\Gamma=\begin{tikzcd}[row sep=0.3cm, column sep=0.3cm]
		 & \bullet \arrow[r,"\alpha"] & \bullet \arrow[r,"1"] & \bullet \arrow[r,"2"] & \bullet \arrow[r,"3"] & \bullet \arrow[r,"\beta"] &\bullet \arrow[dr,"6"] & \\
		  &  &  &  &&& &\bullet \arrow[d,"5"]  \\
		\bullet \arrow[uur,"\gamma"] &  &  &  &&&& \bullet \arrow[d,"4"] \\
		\bullet  \arrow[u,"4"]&  &  & & &&& \bullet \arrow[d,"\gamma"] \\
		\bullet  \arrow[u,"5"]&  &  &  &&&& \bullet  \arrow[d,"7"]\\
		\bullet  \arrow[u,"6"]&  &  &  &&&& \bullet \arrow[ld,"8"] \\
	 & \bullet \arrow[lu,"\beta"]  & \bullet \arrow[l,"3"]  & \bullet \arrow[l,"2"]& \bullet \arrow[l,"1"] & \bullet \arrow[l,"\alpha"]&\bullet \arrow[l,"9"]&
	\end{tikzcd} 
\]
\vspace{0.3cm}

\[
Q=
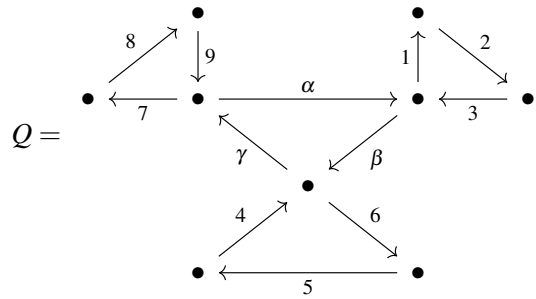
\begin{tikzcd}
	& \bullet \arrow[d,"9"]&  &\bullet \arrow[rd,"2"]& \\
\bullet \arrow[ru,"8"] & \bullet \arrow[l,"7"] \arrow[rr,"\alpha"] & & \bullet \arrow[u,"1"] \arrow[dl,"\beta"] & \bullet \arrow[l,"3"] \\
& & \bullet \arrow[ul,"\gamma"] \arrow[rd,"6"] & & \\
	& \bullet \arrow[ur,"4"] &  &\bullet \arrow[ll,"5"]  &
\end{tikzcd}
\]
Let $c:\Gamma \to Q$ be an winding sending each arrow of $\Gamma$ to the corresponding arrow in $Q$. From \cite[Lemma 5.18]{jun2021coefficient}, $c$ has a nice sequence distinguishing vertices. Let $A=\{7,8,9\}$. Clearly, $c|_A$ is a winding with a nice sequence distinguishing vertices. However, one can easily check that $c/A$ is a winding which does not have a nice sequence distinguishing vertices (again from \cite[Lemma 5.18]{jun2021coefficient}).
\end{myeg}


\bibliography{quiver}\bibliographystyle{alpha}

\begin{thebibliography}{CHWW15}

\bibitem[Bor16]{borger2016witt}
James Borger.
\newblock Witt vectors, semirings, and total positivity.
\newblock {\em Absolute Arithmetic and $\mathbb{F}_1$-Geometry}, pages
  273--329, 2016.

\bibitem[CC10]{con1}
Alain Connes and Caterina Consani.
\newblock Schemes over $\mathbb{ F}_1$ and zeta functions.
\newblock {\em Compositio Mathematica}, 146(6):1383--1415, 2010.

\bibitem[CC11]{con2}
Alain Connes and Caterina Consani.
\newblock On the notion of geometry over $\mathbb{F}_1$.
\newblock {\em Journal of Algebraic Geometry}, 20(3):525--557, 2011.

\bibitem[CCM09]{connes2009fun}
Alain Connes, Caterina Consani, and Matilde Marcolli.
\newblock Fun with $\mathbb{F}_1$.
\newblock {\em Journal of Number Theory}, 129(6):1532--1561, 2009.

\bibitem[CHWW15]{cortinas2015toric}
Guillermo Corti{\~n}as, Christian Haesemeyer, Mark~E Walker, and Charles
  Weibel.
\newblock Toric varieties, monoid schemes and cdh descent.
\newblock {\em Journal f{\"u}r die reine und angewandte Mathematik},
  2015(698):1--54, 2015.

\bibitem[Dei05]{Deitmar}
Anton Deitmar.
\newblock Schemes over $\mathbb{F}_1$.
\newblock In {\em Number fields and function fields-two parallel worlds}, pages
  87--100. Springer, 2005.

\bibitem[Dei08]{deitmar2008f1}
Anton Deitmar.
\newblock $\mathbb{F}_1$-schemes and toric varieties.
\newblock {\em Contributions to Algebra and Geometry}, 49(2):517--525, 2008.

\bibitem[DK19]{dyckerhoff2012higher}
Tobias Dyckerhoff and Mikhail Kapranov.
\newblock {\em Higher {S}egal {S}paces}, volume 2244.
\newblock Springer, 2019.

\bibitem[Fra21]{franzen2021rationality}
Hans Franzen.
\newblock Rationality of rigid quiver {G}rassmannians.
\newblock {\em Algebras and Representation Theory}, pages 1--11, 2021.

\bibitem[GG16]{giansiracusa2016equations}
Jeffrey Giansiracusa and Noah Giansiracusa.
\newblock Equations of tropical varieties.
\newblock {\em Duke Mathematical Journal}, 165(18):3379--3433, 2016.

\bibitem[Gre95]{green1995hall}
James~A Green.
\newblock Hall algebras, hereditary algebras and quantum groups.
\newblock {\em Inventiones mathematicae}, 120(1):361--377, 1995.

\bibitem[Hau12]{haupt2012euler}
Nicolas Haupt.
\newblock Euler characteristics of quiver {G}rassmannians and {R}ingel-{H}all
  algebras of string algebras.
\newblock {\em Algebras and representation theory}, 15(4):755--793, 2012.

\bibitem[Ire11]{irelli2011quiver}
G~Cerulli Irelli.
\newblock Quiver {G}rassmannians associated with string modules.
\newblock {\em Journal of Algebraic Combinatorics}, 33(2):259--276, 2011.

\bibitem[JS20]{jun2020quiver}
Jaiung Jun and Alex Sistko.
\newblock On quiver representations over $\mathbb{F}_1$.
\newblock {\em arXiv preprint arXiv:2008.11304 (to appear in Algebr. Represent.
  Theory)}, 2020.

\bibitem[JS21]{jun2021coefficient}
Jaiung Jun and Alex Sistko.
\newblock Coefficient quivers, $\mathbb{F}_1$-representations, and {E}uler
  characteristics of quiver {G}rassmannians.
\newblock {\em arXiv preprint arXiv:2112.06291}, 2021.

\bibitem[Kin]{kinser2013trees}
Ryan Kinser.
\newblock Tree modules and counting polynomials.
\newblock {\em Algebras and Representation Theory}, 16.

\bibitem[Kra91]{krause1991maps}
Henning Krause.
\newblock Maps between tree and band modules.
\newblock {\em Journal of Algebra}, 137(1):186--194, 1991.

\bibitem[Lor15]{lorscheid2015schubert}
Oliver Lorscheid.
\newblock Schubert decompositions for quiver {G}rassmannians of tree modules.
\newblock {\em Algebra \& Number Theory}, 9(6):1337--1362, 2015.

\bibitem[Rei13]{reineke2013every}
Markus Reineke.
\newblock Every projective variety is a quiver {G}rassmannian.
\newblock {\em Algebras and Representation Theory}, 16(5):1313--1314, 2013.

\bibitem[Rin90]{ringel1990hall}
Claus~Michael Ringel.
\newblock Hall algebras and quantum groups.
\newblock {\em Inventiones mathematicae}, 101(1):583--591, 1990.

\bibitem[Rin98]{ringel1998exceptional}
Claus~Michael Ringel.
\newblock Exceptional modules are tree modules.
\newblock {\em Linear Algebra and its applications}, 275:471--493, 1998.

\bibitem[Sou04]{soule2004varietes}
Chr. Soul{\'e}.
\newblock Les vari{\'e}t{\'e}s sur le corps {\`a} un {\'e}l{\'e}ment.
\newblock {\em Mosc. Math. J}, 4(1):217--244, 2004.

\bibitem[Szc11]{szczesny2011representations}
Matt Szczesny.
\newblock Representations of quivers over $\mathbb{F}_1$ and {H}all algebras.
\newblock {\em International Mathematics Research Notices},
  2012(10):2377--2404, 2011.

\bibitem[Tit56]{tits1956analogues}
Jacques Tits.
\newblock Sur les analogues alg{\'e}briques des groupes semi-simples complexes.
\newblock In {\em Colloque d’algebre sup{\'e}rieure, Bruxelles}, pages
  261--289, 1956.

\bibitem[Wei10]{weist2010tree}
Thorsten Weist.
\newblock Tree modules of the generalised {K}ronecker quiver.
\newblock {\em Journal of Algebra}, 323(4):1107--1138, 2010.

\bibitem[Wei12]{weist2012tree}
Thorsten Weist.
\newblock Tree modules.
\newblock {\em Bullentins of the London Mathematical Society}, 44(5):882--892,
  2012.

\end{thebibliography}

\end{document}